\documentclass[twoside]{article}
\usepackage[utf8]{inputenc}
\usepackage[french,british]{babel}
\usepackage{enumitem}  
\usepackage[T1]{fontenc}
\usepackage{amsmath,amsthm,amsfonts,amssymb} 
\usepackage{bbm} 
\usepackage{dsfont}
\usepackage{graphicx} 
\usepackage{float} 
\usepackage[margin=3cm]{geometry}
\usepackage{hyperref} 
\usepackage{xcolor} 
\usepackage{mathrsfs}
\hypersetup{colorlinks=true,citecolor={blue},
    linkbordercolor={white},
    linkcolor={green!50!black},urlcolor={black}} 
\usepackage{tikz} 
\usetikzlibrary{automata,positioning} 
\usepackage{authblk}
\usepackage{algorithm}
\usepackage{algorithmic}

\newtheorem{theo}{Theorem}[section]
\newtheorem*{theo*}{Theorem}
\newtheorem{prop}[theo]{Proposition}
\newtheorem*{prop*}{Proposition}
\newtheorem{lemme}[theo]{Lemma}
\newtheorem*{lemme*}{Lemma}
\newtheorem{definition}[theo]{Definition}
\newtheorem{corollaire}[theo]{Corollary}
\newtheorem*{corollaire*}{Corollary}
\newtheorem{exemple}[theo]{Example}
\newtheorem{remarque}{Remark}

\newcommand{\PP}{\mathbb{P}}

\newcommand{\NN}{\mathbb{N}}
\newcommand{\CC}{\mathbb{C}}
\newcommand{\ZZ}{\mathbb{Z}}
\newcommand{\RR}{\mathbb{R}}
\newcommand{\XX}{\mathbb{X}}

\newcommand{\tend}[4][]{\overset{#4}{\xrightarrow[#2\to#3]{#1}}}

\newcommand{\NIZ}[1]{\textrm{NIZ}_{_{#1}}}
\newcommand{\CST}[1]{\textrm{CST}_{_{#1}}}

\newcommand{\xdownarrow}[1]{%
  {\left\downarrow\vbox to #1{}\right.\kern-\nulldelimiterspace}
}

\title{On the variation of the sum of digits in the Zeckendorf representation : an algorithm to compute the distribution and mixing properties.}

\author[]{Yohan HOSTEN
}

\affil[]{LAMFA, Université de Picardie Jules Verne, CNRS UMR 7352, 33, rue Saint-Leu, 80000, Amiens, France, \href{mailto:yohan.hosten@math.cnrs.fr}{yohan.hosten@math.cnrs.fr}}
\date{}

\begin{document}
\maketitle
\textbf{Abstract.} We study probability measures defined by the variation of the sum of digits in the Zeckendorf representation. For $r\ge 0$ and $d\in\ZZ$, we consider $\mu^{(r)}(d)$ the density of integers $n\in\NN$ for which the sum of digits increases by $d$ when $r$ is added to $n$. We give a probabilistic interpretation of $\mu^{(r)}$ via the dynamical system provided by the odometer of Zeckendorf-adic integers and its unique invariant measure. We give an algorithm for computing $\mu^{(r)}$ and we deduce a control on the tail of the negative distribution of $\mu^{(r)}$, as well as the formula $\mu^{(F_{\ell})}=\mu^{(1)}$ where $F_{\ell}$ is a term in the Fibonacci sequence. Finally, we decompose the Zeckendorf representation of an integer $r$ into so-called ``blocks'' and show that when added to an adic Zeckendorf integer, the successive actions of these blocks can be seen as a sequence of mixing random variables.

\textbf{Résumé.} On étudie des mesures de probabilités définies grâce à la variation de la somme des chiffres en représentation de Zeckendorf. Pour $r\ge 0$ et $d\in\ZZ$, on considère $\mu^{(r)}(d)$ la densité des entiers $n\in\NN$ pour lesquels la somme des chiffres augmente de $d$ quand on ajoute $r$ à $n$. On donne une interprétation probabiliste à $\mu^{(r)}$ via le système dynamique fourni par l'odomètre des entiers Zeckendorf-adiques et son unique mesure invariante. Nous donnons un algorithme de calcul de $\mu^{(r)}$ et en déduisons un contrôle de la queue de distribution négative de $\mu^{(r)}$ ainsi que la formule $\mu^{(F_{\ell})}=\mu^{(1)}$ où $F_{\ell}$ est un terme de la suite de Fibonacci. Enfin, on décompose l'écriture en représentation de Zeckendorf d'un entier $r$ en ce que l'on appelle des ``blocs'' et nous montrons que lors d'une addition avec un entier Zeckendorf-adique, les actions successives de ces blocs peuvent être vues comme une suite de variables aléatoires mélangeantes.

\vspace{5mm}


\vspace{5mm}


\section{Introduction}
Throughout this article, we denote by $\NN:=\{0,1,2,\cdots\}$ the set of \emph{integers} and $\varphi:=\frac{1+\sqrt{5}}{2}$ represents the \emph{golden ratio}. We also define the well-known \emph{Fibonacci sequence} as follows: 
\begin{align}\label{Fib_def}
    &\left\{ \begin{array}{lll}
     F_1     & :=1,               & \\
     F_2     & :=1,               & \\
     F_{k} & :=F_{k-1}+F_{k-2}     & \textrm{ if } k\ge 3.
\end{array}\right.
\end{align}
By Zeckendorf's theorem (see \cite{EZ}, Theorem I.a, page 179 or \cite{CGL} Theorem A page 1), every integer can be uniquely written as a sum of non-consecutive Fibonacci terms. That is: for any $n\in\NN$, there exists a unique sequence of \emph{digits} $(n_k)_{k\ge 2}\in\{0,1\}^{\infty}$ without two consecutive $1$'s, finitely many of them being equal to $1$ and such that
\begin{align}\label{Zeck_decomposition}
        n&=\sum_{k\ge 2}n_kF_k.
\end{align}

For $n\neq 0$ and $\ell:=\max\{k:~n_k\neq 0\}$, we introduce the notation $\overline{n_{\ell}\cdots n_2}:=n$ which generalises the usual way we write numbers in integer base, and which we refer to as the \emph{(Zeckendorf) expansion} of $n$. This way to expand numbers is actually a particular case of an \emph{Ostrowski's numeration system} (see \cite{AO}, page 1 or \cite{VB}, page 211 or \cite{GBPL}). By convention, we set $\overline{0}:=0$. Then, we define the (Zeckendorf-)\emph{sum-of-digits} function as
\[s(n):=\sum_{k\ge 2} n_k.\]
A central object in our paper is the \emph{variation of the sum of digits} when we add a fixed integer $r$ to  $n$: for $r,n\in\NN$, we set
\begin{equation}\label{def delta zeck}
    \Delta^{(r)}(n):=s(n+r)-s(n).
\end{equation}
The analogous variation in integer base has been studied extensively. The first appearance in the integer base case is in a paper from Bésineau \cite{JB} in 1970. Using a statistical vocabulary, he shows the existence of the asymptotic density for the set of integers such that the variation is some integer $d\in\ZZ$. Given an integer $r$, these densities can be seen as a probability law. The variance of this law is studied, in binary base, by Emme and Prikhod'ko \cite{JEAP} or Spiegelhofer and Wallner \cite{LSMW}. Emme and Hubert \cite{JEPH1} prove a Central Limit Theorem in the binary case which is improved by Spiegelhofer and Wallner \cite{LSMW} (still in binary) or by the author in collaboration with de~la~Rue and Janvresse \cite{TREJYH} (in arbitrary integer base). We also refer to \cite{JFMLS,LS,MDMKLS,LS2} for connected results. Some theorems are proved for the Zeckendorf expansion of an integer: for instance, Griffiths \cite{MG} about the digit proportions; or Labbé and Lepšovà about the addition in this numeration system \cite{SBJL}. Drmota, Müllner and Spiegelhofer obtain results about the existence of primes numbers with a fixed Zeckendorf sum-of-digits (in \cite{MDCMLS}). However, not much has been done about this variation in the Zeckendorf representation except the works from Dekking \cite{FMD} which caracterise the integers such that $\Delta^{(1)}$ is $>0$ (or $<0$ or $=0$) and from Spiegelhofer \cite{LS} (Lemma 1.30) which adapt Bésineau's techniques and prove that, for any $d\in\ZZ$, the following asymptotic density exists
\[\mu^{(r)}(d):=\lim\limits_{N\rightarrow +\infty}\frac{1}{N}\left\vert\left\{ n<N:~\Delta^{(r)}(n)=d\right\}\right\vert.\]
In the present paper, we adapt the ergodic theory point of view introduced in \cite{TREJYH} to the Zeckendorf expansion, recovering Spiegelhofer's result and getting new results about $\mu^{(r)}$ and $\Delta^{(r)}$. In particular, we provide an algorithm that computes $\mu^{(r)}(d)$ and we represent $\Delta^{(r)}$ as the sum of a stochastic mixing process. A perspective we have with this result is to find a Central Limit Theorem for $\Delta^{(r)}$.

To find our results, following the path started in \cite{TREJYH}, we study the variations of the sum-of-digits function in an appropriate probability space given by the compact set $\XX$ of (Zeckendorf-)\emph{adic numbers}. We consider on $\XX$ the action of the odometer, and endow $\XX$ with its unique invariant probability measure $\PP$.

We extend $\Delta^{(r)}$ almost everywhere on $\XX$ and show in Section~\ref{SOzeck} (Proposition~\ref{prop_moment Zeck}) that, for every $d\in\ZZ$
\[\mu^{(r)}(d)=\PP\left(\left\{x\in\XX~:~\Delta^{(r)}(x)=d\right\}\right). \]
Using the Rokhlin towers of the dynamical system, we provide an algorithm to compute $\mu^{(r)}(d)$. This algorithm and its consequences can be adapted in integer base. A first implication is the following corollary on the behaviour of the (negative) tail of the distribution.
\begin{corollaire}\label{corollaire pour algo2}
    For $d$ small enough in $\ZZ$, we have the formula
    \[\mu^{(r)}(d-1)=\mu^{(r)}(d)\times \frac{1}{\varphi^{2}}.\]
\end{corollaire}
\begin{remarque}
One can show that the analogous result in base $b\ge 2$ is the same replacing the formula by $\mu^{(r)}(d-(b-1))=\mu^{(r)}(d)\times \frac{1}{b}$.
\end{remarque}
Another implication is the next theorem about the measure $\mu^{(F_{\ell})}$. 
\begin{theo}\label{theo muFk=mu1}
For any $\ell\ge 3$
\begin{align}\label{muFk=mu1}
    \mu^{(F_{\ell})}=\mu^{(1)}.
\end{align}
\end{theo}
The analogous result in integer base $b$ replaces $F_{\ell}$ by $b^{\ell}$. It is actually a trivial result in base $b$. However, in the Zeckendorf decomposition, this result is much less obvious, due to the particular behaviour of carry propagations that we describe in Subsection~\ref{HTAI}. The main difference with additions in base $b$ is that carries propagate, here, in both directions.

Now, to state our mixing result, we need to define the notion of \emph{blocks} in the expansion of an integer $r$ and to define a probabilistic notion of ($\alpha$-)\emph{mixing coefficients} (see the survey from Bradley \cite{RCB1} for others).
\begin{definition}\label{blocks Zeck}
       A block in the expansion of an integer $r\in\NN$ is defined as a maximal sequence of the pattern $\overline{10}$. (If $r_2=1$, we agree that a maximal sequence $\overline{r_{2\ell}\cdots r_2}$ is a block if $r_{2k}=1$ for $k=1,\cdots, \ell$.) We define $\rho(r)$ as the number of blocks in the expansion of $r$.
\end{definition}
\begin{figure}[H]
    \centering
    \includegraphics[scale=0.35]{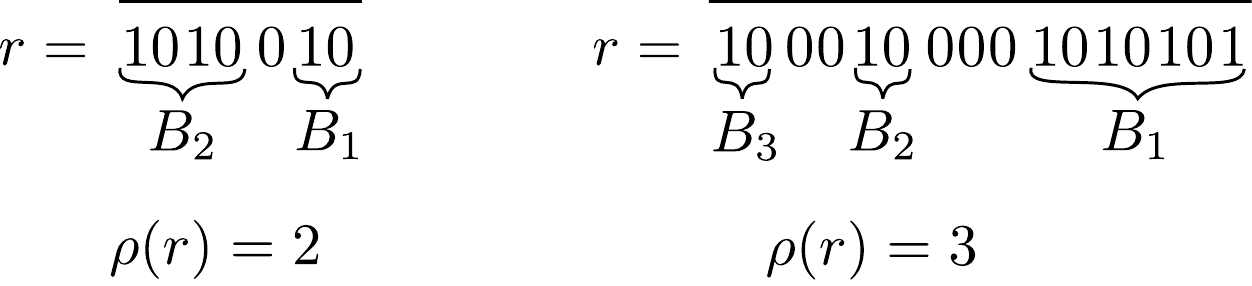}
    \caption{Two examples of decomposition in blocks.}
\end{figure}

\begin{definition}\label{coeff alpha melange Zeck}
      Let $(X_i)_{i\ge 1}$ be a (finite or infinite) sequence of random variables. The associated \emph{$\alpha$-mixing coefficients} $\alpha(k)$, $k\ge 1$, are defined by
   \[\alpha(k):=\sup_{p\ge 1} \hspace{1mm} \sup_{A,B} \hspace{2mm} \lvert \PP(A\cap B)-\PP(A)\PP(B) \rvert\]
   where the second supremum is taken over all events $A$ and $B$ such that 
\begin{itemize}
    \item $A\in\sigma(X_i:1\le  i\le p)$ and
    \item $B\in\sigma(X_i:  i\ge k+p)$.
\end{itemize}
By convention, if $X_i$ is not defined when $i\ge k+p$ then the $\sigma$-algebra is trivial.
\end{definition}
For an integer $r$, we enumerate the blocks of $r$ from the units position and we define $X_i^{(r)}$ as the action of the $i^{th}$ block once the previous blocks have already been taken into consideration (see Subsection~\ref{the process} for more details). These actions are constructed in order to have the equality
\begin{align}\label{decomposition delta zeck}
    \Delta^{(r)}=\sum_{i=1}^{\rho(r)}X_i^{(r)}.
\end{align}
We state the following theorem which gives an upper bound on the $\alpha$-mixing coefficients that is independent of $r$. 
\begin{theo}\label{mixing action bloc zeck}
 The $\alpha$-mixing coefficients of $(X_i^{(r)})_{i=1,\cdots,\rho(r)}$ satisfy
\[\forall k\ge 1, \hspace{5mm} \alpha(k)\le 12\left(1-\frac{1}{\varphi^8}\right)^{\frac{k}{6}}+\frac{1}{\varphi^{2k}}.\]
\end{theo}
In \cite{TREJYH}, a similar result together with some control of the variance depending on the number of blocks is used to prove a Central Limit Theorem for $\mu^{(r)}$ in the integer-base case. Such a control is, here, more difficult than in base $b$ since we do not have inductive relations on $\mu^{(r)}$.

\subsection*{Roadmap}

Section~\ref{HTA} is devoted to understanding the effect of adding ``$1$'' and ``$F_k$'' on the digits of a Z-adic integers. We highlight stopping patterns in the carry propagation, which play an important role in the analysis.

In Section~\ref{UEO}, we focus on the unique ergodic measure $\PP$ of the odometer. We show that $\PP$ satisfies some renewal properties (Proposition~\ref{renouvellement}) and we estimate the $\phi$-mixing coefficients for the coordinates of a Z-adic integer (Proposition~\ref{Melange_coord}).

Then, in Section~\ref{SOzeck}, we place the study of the measures $\mu^{(r)}$ in the context of the odometer on $\XX$. We extend $\Delta^{(r)}$ almost everywhere on $\XX$ and we show that the convergence
\[\lim\limits_{N\to\infty}\frac{1}{N}\sum_{n<N}f(\Delta^{(r)}(n))=\int_{\XX} f(\Delta^{(r)}(x)) \mathrm{d}\PP(x)\]
is satisfied for functions $f:\ZZ\rightarrow \CC$ of polynomial growth (Proposition \ref{prop_moment Zeck}) and, more generally, for functions $f$ such that $f\circ\Delta^{(r)}$ is integrable (Proposition \ref{prop_gene Zeck}). We deduce from Proposition \ref{prop_moment Zeck} that $\mu^{(r)}(d)=\PP(\{x\in \XX: \Delta^{(r)}(x)=d \})$ and that $\mu^{(r)}$ has finite moments. In particular, we show that $\mu^{(r)}$ is of zero-mean.

In Section~\ref{HTCM}, we construct an algorithm that computes $\mu^{(r)}$. A pseudo-code is given in Subsection~\ref{PC}. Also, in Subsection~\ref{ROTA}, we prove Corollary~\ref{corollaire pour algo2} which gives a control on the tail of $\mu^{(r)}$.

Section~\ref{HTPMFM1} is devoted to the proof of Theorem~\ref{theo muFk=mu1}. The proof consists in applying the algorithm.

In the last section, we build a finite sequence of random variables associated to the addition of some integer $r$ that will give the decomposition of $\Delta^{(r)}$ mentioned at \eqref{decomposition delta zeck}. Using that sequence, we prove Theorem~\ref{mixing action bloc zeck} on the estimation of the $\alpha$-mixing coefficients for this sequence.

\subsection*{Acknowledgments}
The author wishes to thank Thierry de la Rue and Élise Janvresse who introduced him to this subject. They gave the author a lot of valuable advice.

\section{How to do additions}\label{HTA}

    \subsection{How to add integers}\label{HTAI}
    
Here we describe the algorithm of the addition using the Zeckendorf way to represent numbers. We start with the addition of $1$. There are two cases.
\begin{enumerate}
    \item Either there exists $\ell\ge 0$ such that the Zeckendorf decomposition of $n$ is 
    \[n=\sum_{k=1}^{\ell}F_{2k+1}+\sum_{k\ge 2\ell+4}n_kF_k\] 
    Then, with the relations \eqref{Fib_def}, we get 
    \[\begin{array}{ccccccc}
         (n)    &       & \cdots    & n_{2\ell+4}   & 0 & 0 & (10)^{\ell} \\
                &+      &           &               &   &   & 1 \\ \hline
         (n+1)  &=      & \cdots    & n_{2\ell+4}   & 0 & 1 & (00)^{\ell}   
    \end{array}\]
    Indeed, $F_3+\cdots F_{2\ell+1}+F_2=F_{2\ell+2}$.
    \item Or there exists $\ell\ge 0$ such that 
    \[n=\sum_{k=1}^{\ell+1}F_{2k}+\sum_{k\ge 2\ell+5}n_kF_k.\]
    Then, for the same reasons 
    \[\begin{array}{ccccccccc}
         (n)    &       & \cdots    & n_{2\ell+5}   & 0 & 0 & 1 & (01)^{\ell} \\
                &+      &           &                &   &   &    & 1 \\ \hline
         (n+1)  &=      & \cdots    & n_{2\ell+5}   & 0 & 1 & 0 & (00)^{\ell}   
    \end{array}\]
\end{enumerate}
We observe that adding $1$ to the rightmost digit of a block as defined in \ref{blocks Zeck} modifies that block into a chain of $0$ digits of the same length and put a $1$ at the first left position of the block.

Of course, in order to compute the addition of two integers, adding $1$ as many times as needed is enough. However, we want to show a main difference with the addition in integer base. In integer base, adding $1$ at some position $k\ge 2$ to an integer $n$ may change the digits of the expansion of $n$ of higher indices due to a carry propagations. Here, it can also change the digits of lower indices. 
We consider the addition $n+F_k$ where $k\ge 3$. We observe that a consequence of \eqref{Fib_def} is 
\begin{align}\label{Fib_relat2}
    2F_k=&\left\{\begin{array}{ll}
     F_2+F_4 &\textrm{ if }k=3 \textrm{,} \\
     F_{k+1}+F_{k-2} & \textrm{ otherwise.}
\end{array}\right.
\end{align}
Many cases appear. For simplicty, the digit in \textcolor{red}{\textbf{color}} will represent the digit at position $k$ and we do not represent digits that remain the same in the expansion of $n$ and $n+F_k$. We start with the cases that change only digits of indices $\ge k$.
\begin{enumerate}
    \item If $n=\overline{\cdots 0\textcolor{red}{\textbf{0}}0\cdots}$ then 
        \[\begin{array}{cccccccc}
         (n)      &       & \cdots    & 0 & \textcolor{red}{\textbf{0}} & 0& \cdots\\
         (+F_k)   &+      &           &   & \textcolor{red}{\textbf{1}} &  &  \\ \hline
         (n+F_k)  &=      & \cdots    & 0 & \textcolor{red}{\textbf{1}} & 0& \cdots
    \end{array}\]
    \item If there exists $\ell\ge 0$ such that $n=\overline{\cdots 001(01)^{\ell}\textcolor{red}{\textbf{0}}\cdots}$ then 
    \[\begin{array}{cccccccccc}
         (n)        &       & \cdots    & 0 & 0 & 1 & (01)^{\ell} &\textcolor{red}{\textbf{0}} &  \cdots\\
         (+F_k)     &+      &           &   &   &   &             &\textcolor{red}{\textbf{1}} &   \\ \hline
         (n+F_k)    &=      & \cdots    & 0 & 1 & 0 &(00)^{\ell}  &\textcolor{red}{\textbf{0}} &  \cdots
    \end{array}\]
\end{enumerate}
We continue with the first case where the digit of index $k-1$ is changed.

\begin{enumerate}[resume]
     \item If there exists $\ell\ge 0$ such that $n=\overline{\cdots 00(10)^{\ell}\textcolor{red}{\textbf{0}}1\cdots}$ then \[\begin{array}{ccccccccc}
         (n)      &       & \cdots    & 0 & 0 & (10)^{\ell} &\textcolor{red}{\textbf{0}} & 1 & \cdots\\
         (+F_k)   &+      &           &   &   &             &\textcolor{red}{\textbf{1}} &   &  \\ \hline
         (n+F_k)  &=      & \cdots    & 0 & 1 & (00)^{\ell} &\textcolor{red}{\textbf{0}} & 0 & \cdots
    \end{array}\]
\end{enumerate}

Now, we consider the cases where many digits of indices $<k$ are changed. It is due to \eqref{Fib_relat2}.
\begin{enumerate}[resume]
    \item\label{item 4} If $k\ge 5$ and there exist $\ell,\ell'\ge 0$ such that $n=\overline{\cdots 00(10)^{\ell}\textcolor{red}{\textbf{1}}(01)^{\ell'}000\cdots}$ then
        \[\begin{array}{cccccccccccc}
         (n)      &       & \cdots    & 0 & 0 & (10)^{\ell} &\textcolor{red}{\textbf{1}} & (01)^{\ell'} & 0 & 0 & 0 & \cdots\\
        (+F_k)    &+      &           &   &   &             &\textcolor{red}{\textbf{1}} &              &   &   &   & \\ \hline
         (n+F_k)  &=      & \cdots    & 0 & 1 & (00)^{\ell} &\textcolor{red}{\textbf{0}} & (10)^{\ell'} & 0 & 1 & 0 & \cdots
    \end{array}\]
    \item\label{item 5} If $k\ge 6$ and there exist $\ell,\ell'\ge 0$ such that $n=\overline{\cdots 00(10)^{\ell}\textcolor{red}{\textbf{1}}(01)^{\ell'}0010\cdots}$ then
        \[\begin{array}{ccccccccccccc}
         (n)      &       & \cdots    & 0 & 0 & (10)^{\ell} &\textcolor{red}{\textbf{1}} & (01)^{\ell'} & 0 & 0 & 1 & 0 &\cdots\\
         (+F_k)   &+      &           &   &   &             &\textcolor{red}{\textbf{1}} &              &   &   &   &   & \\ \hline
         (n+F_k)  &=      & \cdots    & 0 & 1 & (00)^{\ell} &\textcolor{red}{\textbf{0}} & (10)^{\ell'} & 1 & 0 & 0 & 0 &\cdots
    \end{array}\]
\end{enumerate}
Finally, we consider the ``boundary'' cases where all the digits of indices $<k$ are changed.
\begin{enumerate}[resume]
    \item If $k\ge 4$, $k$ is even and there exists $\ell\ge 0$ such that $n=\overline{\cdots 00(10)^{\ell}\textcolor{red}{\textbf{1}}(01)^{\frac{k-2}{2}}}$ then
        \[\begin{array}{cccccccc}
         (n)      &       & \cdots    & 0 & 0 & (10)^{\ell} &\textcolor{red}{\textbf{1}} & (01)^{\frac{k-2}{2}}\\
        (+F_k)    &+      &           &   &   &             &\textcolor{red}{\textbf{1}} & \\ \hline
         (n+F_k)  &=      & \cdots    & 0 & 1 & (00)^{\ell} &\textcolor{red}{\textbf{0}} & (10)^{\frac{k-2}{2}}
    \end{array}\]
    \item If $k\ge 5$, $k$ is odd and there exists $\ell\ge 0$ such that $n=\overline{\cdots 00(10)^{\ell}\textcolor{red}{\textbf{1}}(01)^{\frac{k-3}{2}}0}$ then
        \[\begin{array}{ccccccccc}
         (n)      &       & \cdots    & 0 & 0 & (10)^{\ell} &\textcolor{red}{\textbf{1}} & (01)^{\frac{k-3}{2}} & 0 \\
        (+F_k)    &+      &           &   &   &             &\textcolor{red}{\textbf{1}} &                      &\\ \hline
         (n+F_k)  &=      & \cdots    & 0 & 1 & (00)^{\ell} &\textcolor{red}{\textbf{0}} & (10)^{\frac{k-3}{2}} & 1  
    \end{array}\]
\end{enumerate}

    \subsection{The Z-adic integers and how to add one of them to an integer}\label{ZIHTAI}

We define the \textit{Zeckendorf-adic integers} (or \textit{Z-adic integers} for simplicity) as elements of \[\XX:=\big\{x\in\{0;1\}^{\NN_{\ge 2}} : \forall k\ge 2~x_kx_{k+1}=0\big\}.\]
Also, coordinates of a Z-adic integer $x\in \XX$ are interpreted as digits in the Zeckendorf representation: elements of $\XX$ can be viewed as ``generalized integers having possibly infinitely many non zero digits in Zeckendorf representation''. An element $x=(x_k)_{k\ge 2}\in \XX$ will be represented as a left-infinite sequence $(\cdots, x_3, x_2)$, $x_2$ being the unit digit. We endow $\XX$ with the product topology which turns it into a compact metrizable space. The set $\NN$ can be identified with the subset of sequences with finite support. More precisely, using the inclusion function 
\begin{align*}
    i:n=\overline{n_{\ell}\cdots n_2} \in\NN \longmapsto (\cdots, 0,n_{\ell},\cdots, n_2) \in \XX
\end{align*}
we identify $\NN$ and $i(\NN)$. We can also identify the notation 
\[(\cdots, x_3, x_2)=\overline{\cdots x_3 x_2.}\]
We take the opportunity to define $\XX_f$ as the finite sequences of $0$'s and $1$'s without two consecutive $1$'s. For instance, the Zeckendorf expansion of a given integer is composed using a sequence in $\XX_f$. Let us define, for $\ell\ge 2$ and $(n_k)_{k\ge 2}\in\XX_f$, the \textit{cylinder} $C_{n_{\ell}\cdots n_2}$ as the set of sequences $x\in\XX$ such that $x_i=n_i$ for $ i=2,\cdots, \ell$. We observe that $n_{\ell}\cdots n_2$ is not necessarily the Zeckendorf expansion of a given integer : the leftmost digit(s) can be $0$('s).

We want to extend on $\XX$ the transformation $n\mapsto n+1$ defined on $\NN$. Thanks to the description given in Subsection \ref{HTAI}, it is convenient to consider the transformation $T$ on $\XX$ defined by the following formula, where $\ell\in\NN$
\[\begin{array}{clcl}
    T:  & \XX & \longrightarrow &\XX \\
        & \overline{\cdots x_{2\ell+4}00(10)^{\ell}}& \longmapsto & \overline{\cdots x_{2\ell+4}01(00)^{\ell}} \\
        & \overline{\cdots x_{2\ell'+3}001(01)^{\ell}}& \longmapsto & \overline{\cdots x_{2\ell'+3}010(00)^{\ell}} \\
        & \overline{(10)^{\infty}}& \longmapsto & \overline{0^{\infty}} \\
        & \overline{(01)^{\infty}}& \longmapsto & \overline{0^{\infty}} 
\end{array}\]
Indeed, thanks to the Subsection \ref{HTAI}, we observe that $T_{|\NN}(n)=n+1$. Now, if we take $x\in\XX$ whose expansion contains two consecutive $0$'s, we observe that the sequence $(\overline{x_{\ell}\cdots x_2}+1)_{\ell\ge 2}$ converges to $T(x)$ because the digits will not changed eventually so we can define $x+1:=T(x)$ in that case. Otherwise, if $x\in\XX$ does not have two consecutive $0$'s in its expansion, there are two cases : $\overline{(10)^{\infty}}$ and $\overline{(01)^{\infty}}$. Adding $1$ to the truncated sequence $\left(\overline{(10)^{\ell}}\right)_{\ell\in\NN}$ converges to $\overline{(10)^{\infty}}+1:=\overline{0^{\infty}}$. It is the same for $\overline{(01)^{\infty}}$. Thus, the transformation $T$ can be described in a simpler way as 
 \[T:
\left\lbrace
  \begin{array}{ccc}
    \XX & \longrightarrow   &\XX \\
    x   & \longmapsto       & x+1
  \end{array}
\right.\]

Due to the two pre-images of $0^{\infty}$, $T$ is not a homeomorphism on $\XX$ but remains continuous and surjective on $\XX$. It is one-to-one on $\XX\backslash{\{0^{\infty}\}}$. Thus, $(\XX,T)$ is a topological dynamical system that we call the \emph{Odometer}. 


We know how to add $1$ to a Z-adic integer $x$. Repeating this operation enables us to add an integer $r$ to $x$. For later purposes, we need to specify how to add $F_k$ ($\ge 3$) directly. In Subsection~\ref{HTAI}, we have compute $x+F_k$ for the $x\in\XX$ which has two consecutive $0$'s at indices $>k$. Thus, we only need to focus on what happens if $x$ does not have two consecutive $0$'s at indices $>k$. There is only a finite number of cases to consider which we detail below. For simplicity again, we write in \textcolor{red}{\textbf{color}} the digits at position $k$ and we do not represent digits that are not modified. We start with the cases where the only digits that change are those of indices $\ge k-1$.
\begin{enumerate}
    \item If $x=\overline{(10)^{\infty}\textcolor{red}{\textbf{0}}1\cdots}$ then 
    \[\begin{array}{cccccc}
        (x)    &   & (10)^{\infty} & \textcolor{red}{\textbf{0}} & 1 & \cdots  \\
        (F_k)  & + &               & \textcolor{red}{\textbf{1}} &   &       \\ \hline
        (x+F_k)& = & (00)^{\infty} & \textcolor{red}{\textbf{0}} & 0 & \cdots
    \end{array}\]
    \item If $x=\overline{(01)^{\infty}\textcolor{red}{\textbf{0}}\cdots}$ then 
    \[\begin{array}{cccccc}
        (x)    &   & (01)^{\infty} & \textcolor{red}{\textbf{0}} &  \cdots  \\
        (F_k)  & + &               & \textcolor{red}{\textbf{1}} &        \\ \hline
        (x+F_k)& = & (00)^{\infty} & \textcolor{red}{\textbf{0}} &  \cdots
    \end{array}\]
\end{enumerate}
Now we consider cases where digits of small indices are changed but not all them.
\begin{enumerate}[resume]
    \item If $k\ge 5$ and there exists $\ell'\ge0$ such that $x=\overline{(10)^{\infty}\textcolor{red}{\textbf{1}}(01)^{\ell'}000\cdots}$ then 
    \[\begin{array}{ccccccccc}
        (x)    &   & (10)^{\infty} & \textcolor{red}{\textbf{1}} & (01)^{\ell'} & 0 & 0 & 0 & \cdots  \\
        (F_k)  & + &               & \textcolor{red}{\textbf{1}} &              &   &   &   & \\ \hline
        (x+F_k)& = & (00)^{\infty} & \textcolor{red}{\textbf{0}} & (10)^{\ell'} & 0 & 1 & 0 &  \cdots
    \end{array}\]
    \item If $k\ge 6$ and there exists $\ell'\ge0$ such that $x=\overline{(10)^{\infty}\textcolor{red}{\textbf{1}}(01)^{\ell'}0010\cdots}$ then 
    \[\begin{array}{cccccccccc}
        (x)    &   & (10)^{\infty} & \textcolor{red}{\textbf{1}} & (01)^{\ell'} & 0 & 0 & 1 & 0 & \cdots  \\
        (F_k)  & + &               & \textcolor{red}{\textbf{1}} &              &   &   &   &   & \\ \hline
        (x+F_k)& = & (00)^{\infty} & \textcolor{red}{\textbf{0}} & (10)^{\ell'} & 1 & 0 & 0 & 0 & \cdots
    \end{array}\]
\end{enumerate}
Finally, we consider cases where the whole prefix of $x$ is modified.
\begin{enumerate}[resume]
    \item If $k\ge 4$ and is even and $x=\overline{(10)^{\infty}\textcolor{red}{\textbf{1}}(01)^{\frac{k-2}{2}}}$ then 
    \[\begin{array}{ccccc}
        (x)     &   & (10)^{\infty} & \textcolor{red}{\textbf{1}} & (01)^{\frac{k-2}{2}} \\
        (F_k)   & + &               & \textcolor{red}{\textbf{1}} & \\ \hline
        (x+F_k) & = & (00)^{\infty} & \textcolor{red}{\textbf{0}} & (10)^{\frac{k-2}{2}}
    \end{array}\]
    \item If $k\ge 3$ and is odd and $x=\overline{(10)^{\infty}\textcolor{red}{\textbf{1}}(01)^{\frac{k-3}{2}}0}$ then 
    \[\begin{array}{cccccc}
        (x)     &   & (10)^{\infty} & \textcolor{red}{\textbf{1}} & (01)^{\frac{k-3}{2}} & 0 \\
        (F_k)   & + &               & \textcolor{red}{\textbf{1}} &                      & \\ \hline
        (x+F_k) & = & (00)^{\infty} & \textcolor{red}{\textbf{0}} & (10)^{\frac{k-3}{2}} & 1
    \end{array}\]
\end{enumerate}
We can sum up all these cases in the following proposition.

\begin{prop}\label{T^{F_n}}
The table below summarises the action of the addition of $F_k$ ($k\ge 2)$ on the digits of $x\in\XX$. Here $\ell,\ell'\in\NN$ and we represent in \textcolor{red}{\textbf{color}}, the digits at position $k$. The digits that are not explicitly written remain untouched by the operation.
\begin{align}
    T^{F_k}:    &\XX &\longrightarrow  \hspace{2mm}  & \XX \notag \\
                &\overline{\cdots 0\textcolor{red}{\textbf{\emph{0}}}0\cdots}& \longmapsto \hspace{2mm} & \overline{\cdots 0\textcolor{red}{\textbf{\emph{1}}}0\cdots} \\
                &\overline{\cdots 001(01)^{\ell}\textcolor{red}{\textbf{\emph{0}}}\cdots} & \longmapsto \hspace{2mm}  & \overline{\cdots 010(00)^{\ell}\textcolor{red}{\textbf{\emph{0}}}\cdots} \label{avancee1}\\
                &\overline{\cdots00(10)^{\ell}\textcolor{red}{\textbf{\emph{0}}}1\cdots}&\longmapsto \hspace{2mm} &\overline{\cdots 01(00)^{\ell}\textcolor{red}{\textbf{\emph{0}}}0\cdots} \label{avancee2}\\
                &\overline{\cdots00(10)^{\ell}\textcolor{red}{\textbf{\emph{1}}}(01)^{\ell'}000\cdots}&\longmapsto \hspace{2mm} &\overline{\cdots 01(00)^{\ell}\textcolor{red}{\textbf{\emph{0}}}(10)^{\ell'}010\cdots} \label{recul1} \\
                &\overline{\cdots 00(10)^{\ell}\textcolor{red}{\textbf{\emph{1}}}(01)^{\ell'}0010\cdots}&\longmapsto \hspace{2mm} &\overline{\cdots 01(00)^{\ell}\textcolor{red}{\textbf{\emph{0}}}(10)^{\ell'}1000\cdots} \label{recul2}\\
                &\overline{\cdots 00(10)^{\ell}\textcolor{red}{\textbf{\emph{1}}}(01)^{\frac{k-5}{2}}001}&\longmapsto \hspace{2mm} &\overline{\cdots 01(00)^{\ell}\textcolor{red}{\textbf{\emph{0}}}(10)^{\frac{k-5}{2}}100} \label{11}\\
                &\overline{\cdots 00(10)^{\ell}\textcolor{red}{\textbf{\emph{1}}}(01)^{\frac{k-4}{2}}00}&\longmapsto \hspace{2mm} &\overline{\cdots 01(00)^{\ell}\textcolor{red}{\textbf{\emph{0}}}(10)^{\frac{k-4}{2}}01}\label{12}\\
                &\overline{\cdots 00(10)^{\ell}\textcolor{red}{\textbf{\emph{1}}}(01)^{\frac{k-2}{2}}}&\longmapsto \hspace{2mm} &\overline{\cdots 01(00)^{\ell}\textcolor{red}{\textbf{\emph{0}}}(10)^{\frac{k-2}{2}}} \label{recul max} \\
                &\overline{\cdots 00(10)^{\ell}\textcolor{red}{\textbf{\emph{1}}}(01)^{\frac{k-3}{2}}0}&\longmapsto \hspace{2mm} &\overline{\cdots 01(00)^{\ell}\textcolor{red}{\textbf{\emph{0}}}(10)^{\frac{k-3}{2}}1}\label{17} \\
                &\overline{(01)^{\infty}\textcolor{red}{\textbf{\emph{0}}}\cdots}&\longmapsto \hspace{2mm} &\overline{0^{\infty}\textcolor{red}{\textbf{\emph{0}}}\cdots} \tag{7.bis} \label{avancee infinie}\\
                 &\overline{(10)^{\infty}\textcolor{red}{\textbf{\emph{0}}}1\cdots}&\longmapsto \hspace{2mm} &\overline{0^{\infty}\textcolor{red}{\textbf{\emph{0}}}0\cdots} \tag{8.bis} \\
                &\overline{(10)^{\infty}\textcolor{red}{\textbf{\emph{1}}}(01)^{\ell'}000\cdots}&\longmapsto \hspace{2mm} &\overline{0^{\infty}\textcolor{red}{\textbf{\emph{0}}}(10)^{\ell'}010\cdots} \tag{9.bis}\label{recul3}\\
                &\overline{(10)^{\infty}\textcolor{red}{\textbf{\emph{1}}}(01)^{\ell'}0010\cdots}&\longmapsto \hspace{2mm} &\overline{0^{\infty}\textcolor{red}{\textbf{\emph{0}}}(10)^{\ell'}1000\cdots} \tag{10.bis}\label{recul4}\\
                &\overline{(10)^{\infty}\textcolor{red}{\textbf{\emph{1}}}(01)^{\frac{k-5}{2}}001}&\longmapsto \hspace{2mm} &\overline{0^{\infty}\textcolor{red}{\textbf{\emph{0}}}(10)^{\frac{k-5}{2}}100}\tag{11.bis}\label{14.bis} \\
                &\overline{(10)^{\infty}\textcolor{red}{\textbf{\emph{1}}}(01)^{\frac{k-4}{2}}00}&\longmapsto \hspace{2mm} &\overline{ 0^{\infty}\textcolor{red}{\textbf{\emph{0}}}(10)^{\frac{k-4}{2}}01}\tag{12.bis}\label{15.bis} \\                
                &\overline{(10)^{\infty}\textcolor{red}{\textbf{\emph{1}}}(01)^{\frac{k-2}{2}}}&\longmapsto \hspace{2mm} &\overline{0^{\infty}\textcolor{red}{\textbf{\emph{0}}}(10)^{\frac{k-2}{2}}} \tag{13.bis}\label{16.bis}\\
                &\overline{(10)^{\infty}\textcolor{red}{\textbf{\emph{1}}}(01)^{\frac{k-3}{2}}0}&\longmapsto \hspace{2mm} &\overline{0^{\infty}\textcolor{red}{\textbf{\emph{0}}}(10)^{\frac{k-3}{2}}1} \tag{14.bis}\label{17.bis}
\end{align}
(We precise that $k$ is even in cases \eqref{12}, \eqref{recul max}, \eqref{15.bis} and \eqref{16.bis} and $k$ is odd in cases \eqref{11}, \eqref{17}, \eqref{14.bis} and \eqref{17.bis}. Also, in some cases, $k$ is considered sufficiently large so the operation is possible.)
\end{prop}
    
    \subsection{Stopping conditions when adding an integer to an adic number}\label{subsection Focus}

    Through the cases described in Proposition~\ref{T^{F_n}}, we observe that if there is a $1$ at position $k$ in the expansion of $x$, the addition of $F_k$ yields a carry propagation in both directions:
\begin{itemize}
    \item to the left, modifying digits of higher indices (as in integer base) 
    \item to the right, modifying digits of lower indices.
\end{itemize}
The propagation (in both directions) happens through a maximal sequence of alternative $0$'s and $1$'s and is stopped at the first occurrence of two consecutive $0$'s. But the modifications depend on the propagation direction.
\begin{itemize}
    \item In the propagation to the left, the maximal subword of alternative $1$'s and $0$'s will be transformed into a subword of $0$'s of the same length (case \eqref{recul1} for instance), and the stopping pattern $00$ is transformed into $01$. We precise this propagation also happens if $x_k=0$ (case \eqref{avancee1} for instance).
    \item In the propagation to the right, which only happens if $x_k=1$, the maximal subword of alternative $1$'s and $0$'s is transformed into a symmetrical subword where the $1$'s become $0$'s and vice-versa (cases \eqref{recul1} and \eqref{recul2} for instance). Then the first occurrence of $00$ (in the sense the largest index $\le k$ such that the digits of $x$ are $00$) can either be part of the pattern $w_0:=\overline{01000}$ or $w_1:=\overline{10010}$. (We call $w_0$ and $w_1$ the \emph{right-stopping pattern}.) 
\end{itemize}
Depending on the right-stopping pattern ($\overline{01000}$ or $\overline{10010}$), the modifications of digits at these indices are given by the next scheme.
\begin{figure}[H]
    \centering
    \includegraphics[scale=0.35]{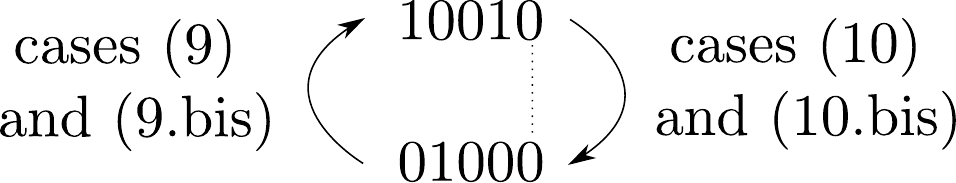}
    \caption{Modifications at the position of the right-stopping pattern.}
    \label{Scheme}
\end{figure}
Note that in both cases, we get a new right-stopping pattern at the same position as before the addition of $F_k$. This addition of $F_k$ to $x\in\XX$ modifies some digits at positions $\le k-2$ only if $x_k=1$, and, in this case, the modification of the digits takes place up to the first occurrence of one of the right-blocking patterns. 

Formally: let $x\in\XX$ and $k\ge 2$.
\begin{itemize}
    \item If $x_k=0$ then $(x+F_k)_n=x_n$ for all $n\le k-2$.
    \item If $x_k=1$ and if there exists $j\le k+1$ such that $x_jx_{j-1}\cdots x_{j-4}$ is a right-stopping patterns $w_i$ ($i\in\{0,1\})$, -we denote by $j'$ the largest index with this property- then 
    \begin{itemize}
        \item $(x+F_k)_n=x_n$ for all $n\le j'-4$
        \item $w_{1-i}$ appears in $(x+F_k)$ at the same position $j'$, unless $w_i=w_0$ and $j'=k+1$, in which case we might have $00010$ instead of $w_1$ at position $j'$ in $x+F_k$ (case~\eqref{recul1} with $\ell=\ell'=0$).
    \end{itemize}
\end{itemize}
A straightforward consequence is the following Proposition.
\begin{prop}
Assume that the right-stopping pattern $w_i$ ($i=0,1$) appears in $x\in\XX$ at position $j\ge 5$, that is: $x_jx_{j-1}\cdots x_{j-4}=w_i$. Let $k\ge j-1$. Then
\begin{itemize}
    \item for all $n\le j-4$, $(x+F_k)_n=x_n$ and
    \item $w_0$, $w_1$ or $00010$ appears in $x+F_k$ at some position $j'$ with $k+1\ge j'\ge j$.
\end{itemize}
\end{prop}
We now state the following corollary that enhances that property when we add not only a Fibonacci term but an integer whose expansion involves Fibonacci numbers of high indices.
\begin{corollaire}\label{blocage recul}
    Let $x\in\XX$. Assume that, for some $\ell\ge 2$, $x_{\ell}=x_{\ell+1}=0$. Let $r\in\NN$ be such that $r_j=0$ for $j=2,\cdots,\ell+1$.
    Then, for each $n\le\ell-2$, we have $(x+r)_n=x_n$.
\end{corollaire}
\begin{proof}
    We are considering the following addition:
    \[
    \begin{array}{ccccccccccc}
        (x)     &       & \cdots    & x_{\ell+3}    & x_{\ell+2}    & 0     & 0     & x_{\ell-1}    & x_{\ell-2}    & x_{\ell-3}    & \cdots   \\
        (r)     & +     & \cdots    & r_{\ell+3}    & r_{\ell+2}    & 0     & 0     & 0             & 0             & 0             & \cdots  \\ \hline
        (x+r)   & =     & \cdots    & \star         & \star         & \star & \star & \star         & x_{\ell-2}    & x_{\ell-3}    & \cdots 
    \end{array}
    \]
    Let $r=F_{k_{s(r)}}+\cdots +F_{k_1}$ be the Zeckendorf decomposition of $r$ with $k_{s(r)}>\cdots>k_2>k_1\ge \ell+2$.
    We first consider the addition of $F_{k_1}$ to $x$:
    \begin{itemize}
        \item if, for all $j$ such that $k_1\ge j\ge \ell$, we have $x_j=0$, then we have for every $n\le\ell+1$, $(x+F_{k_1})_n=x_n$ (in particular $00$ appears at the same place in $x+F_{k_1}$)
        \item otherwise, there exists a largest integer $j'$ with $k_1+1\ge j'\ge \ell+2$, such that one of the right-stopping pattern appears in $x$ at position $j'$. We can then apply the above proposition which proves that 
        \begin{itemize}
            \item either $w_0$ or $w_1$ appears at position $j'$ in $x+F_{k_1}$
            \item or $00$ appears at position $j'$ in $x+F_{k_1}$, and $j'\ge \ell+3$
        \end{itemize}
        In each case, we still have $(x+F_{k_1})_n=x_n$ for $n\le\ell-2$.
    \end{itemize}
    Then we prove by induction on $t$ such that for each $t$, $1\le t\le s(r)$, the above is true for $x+F_{k_1}+\cdots+F_{k_t}$.
\end{proof}

The following lemma ensures that the pattern $00$ is a left-stopping condition: it stops the propagation of a carry coming from the right.
\begin{lemme}\label{blocage_avancee}
    Let $x\in\XX$, $r\in\NN$. Let $\ell\ge 2$ be such that $r<F_{\ell+1}$ and assume $x_{\ell+2}=x_{\ell+3}=0$. Then, for all $k\ge \ell+3$, we have $(x+r)_k=x_k$. 
\end{lemme}
\begin{proof}
The assumption $r<F_{\ell+1}$, implies $r=\overline{r_{\ell}\cdots r_2}$ with $(r_{\ell},\cdots,r_2)\in\XX_f$ . Since $\overline{x_{\ell+1}\cdots x_2}+r<F_{\ell+2}+F_{\ell+1}=F_{\ell+3}$, a carry cannot propagate on digits of indices $\ge \ell+3$: we have the addition 
\[\begin{array}{cccccccccc}
    (x)  &   & \cdots & x_{\ell+4} & 0 & 0              & x_{\ell+1}   & x_{\ell}    & \cdots & x_2 \\
    (r)  & + &        &            &   &                &              & r_{\ell}    & \cdots & r_2 \\ \hline
    (x+r)& = & \cdots & x_{\ell+4} & 0 & (x+r)_{\ell+2} &(x+r)_{\ell+1}& (x+r)_{\ell}& \cdots & (x+r)_2
\end{array}\]
\end{proof}
The next lemma enhances the previous one: it shows that, given $x$ with some restrictions, right-stopping patterns can appear in the expansion $x+F_k$ when $k$ is a small integer. 
\begin{lemme}\label{blocage avancee + motifs}
Let $x\in\XX$ and $r\in\NN$. Let $\ell\ge 2$ be such that $x_{\ell+1}=x_{\ell+2}=x_{\ell+3}=x_{\ell+4}=0$ and $r_{\ell+2}=1$. We also suppose $r<F_{\ell+3}$. Then, for $k=1,2,3,4$
\[(x+r)_{\ell+k}=r_{\ell+k}\]
or
\[(x+r)_{\ell+1}=0 \textrm{ and } (x+r)_{\ell+3}=1.\]
\end{lemme}
\begin{proof}
    We are actually considering the following addition
\[    \begin{array}{ccccccccccc}
        (x) &   & \cdots & x_{\ell+5}   & 0  & 0  & 0   & 0 & x_{\ell} &  \cdots    & x_2 \\
         (r)& + & \cdots &              &    &    & 1   & 0 & r_{\ell} & \cdots     & r_2
    \end{array}
 \]
We want to show that the expansion of $x+r$ is either
   \begin{align}\label{case 1}
      \overline{\cdots x_{\ell+5}0010(x+r)_{\ell}\cdots(x+r)_2} \tag{C1}
  \end{align}
 or
  \begin{align}\label{case 2}
      \overline{\cdots x_{\ell+5}0100(x+r)_{\ell}\cdots(x+r)_2}. \tag{C2}
 \end{align}
 Thanks to Lemma~\ref{blocage_avancee}, we have $(x+r)_k=x_k$ for every $k\ge\ell+4$ which means that the digits of $x+r$ of indices $\le \ell+3$ are given by the addition $r+\overline{x_{\ell}\cdots x_2}$.
 Now, we claim
 \[(x+r)_{\ell+3} \textrm{ or } (x+r)_{\ell+2} \textrm{ is }1.\]
 Indeed, if it is not the case then $\overline{x_{\ell}\cdots x_2}+r<F_{\ell+2}$ while $r\ge F_{\ell+2}$ (since $r_{\ell+2}=1$).
 We consider now both possibilities.
 \begin{itemize}
     \item If $(x+r)_{\ell+2}=1$ then we obtain \eqref{case 1}.
     \item If $(x+r)_{\ell+3}=1$ then we must have $(x+r)_{\ell+1}=0$ since we have
     \begin{align*}
         \overline{x_{\ell}\cdots x_2}+r-F_{\ell+3}&=\overline{x_{\ell}\cdots x_2}+\overline{r_{\ell}\cdots r_2}+F_{\ell+2}-F_{\ell+3}\\
         &< F_{\ell+1}.
     \end{align*}
     Thus, the expansion of $x+r$ is \eqref{case 2}.
 \end{itemize}
\end{proof}

We deduce the next corollary.
\begin{corollaire}\label{stopping condition bloc 1}
Let $x\in\XX$ and $r\in\NN$ such that it exists $\ell\ge 2$ with $x_{\ell}=\cdots=x_{\ell+5}=0$, $r_{\ell+1}=r_{\ell+5}=0$ and $r_{\ell+3}=1$. Denote $\widetilde{r}:=\overline{r_{\ell}\cdots r_2}$. Then, for any $k\ge \ell+4$
    \begin{align}\label{ccl 1 pour bloc de taille 1}
      (x+\widetilde{r}+F_{\ell+3})_k=(x+\widetilde{r})_k=x_k  
    \end{align}
    and, for any $k\le \ell+2$
    \begin{align}\label{ccl2 pour bloc de taille 1}
      (x+r)_k=(x+\widetilde{r}+F_{\ell+3})_k=(x+\widetilde{r})_k.  
    \end{align}
\end{corollaire}
\begin{remarque}\label{rem stop cdts 1}
The hypothesis means that we are considering the following addition 
\[\begin{array}{ccccccccccccc}
    (x) &  & \cdots & x_{\ell+6} & 0 & 0 & 0 & 0 & 0 & 0      & x_{\ell-1}  & \cdots    & x_2  \\
    (r) & +& \cdots & r_{\ell+6} & 0 & 0 & 1 & 0 & 0 & r_\ell     & \cdots  & \cdots    & r_2 \\
\end{array}\]
and the conclusion ensures that the digits of indices $\le \ell+2$ (\textrm{i.e.} those on the right-hand side of the pattern we have imposed) of $x+r$ are the same if we compute this previous addition as if we compute 
\[\begin{array}{ccccccccccccc}
    (x)                     &  & \cdots & x_{\ell+6}   & 0 & 0 & 0 & 0 & 0 & 0      & x_{\ell-1}  & \cdots    & x_2  \\
    (\widetilde{r}+F_{\ell+3}) & +&        &           &   &   & 1 & 0 & 0 & r_\ell     & \cdots  & \cdots    & r_2 \\
\end{array}\]
or if we compute
\[\begin{array}{ccccccccccccc}
    (x)            &  & \cdots & x_{\ell+6}    & 0 & 0 & 0 & 0 & 0 & 0      & x_{\ell-1}  & \cdots    & x_2  \\
    (\widetilde{r})& +&        &            &   &   &   &   &   & r_\ell     & \cdots  & \cdots    & r_2 \\
\end{array}\]

In other words, the corollary states that the conditions we put on $x$ and $r$ stop the propagation of carries in both direction. 
\end{remarque}
\begin{proof}
    First, we have the following addition
    \[\begin{array}{ccccccccccccc}
    (x)                &  & \cdots & x_{\ell+6}    & 0   & 0 & 0 & 0 & 0                         & 0     & x_{\ell-1}   & \cdots    & x_2  \\
    (\widetilde{r})    & +&       &             &     &   &   &   &                           & r_\ell   & \cdots    & \cdots    & r_2 \\ \hline
    (x+\widetilde{r})  & =& \cdots&  x_{\ell+6}    & 0  & 0 & 0 & 0 & (x+\widetilde{r})_{\ell+1}   & \cdots& \cdots    & \cdots    & (x+\widetilde{r})_2 
\end{array}\]
    Indeed, due to Lemma~\ref{blocage_avancee}, we have that $(x+\widetilde{r})_j=x_j$ for every $j\ge \ell+2$. Then, we add $F_{\ell+3}$, it gives $x+\widetilde{r}+F_{\ell+3}$ whose Zeckendorf expansion is
    \[(x+\widetilde{r}+F_{\ell+3})=\overline{\cdots x_{\ell+6}0010(x+\widetilde{r})_{\ell+1}\cdots(x+\widetilde{r})_2}.\]
    We thus get the relation~\eqref{ccl 1 pour bloc de taille 1}.
    
    We now consider the addition (denoting $\widetilde{x}:=x+\widetilde{r}+F_{\ell+3}$)
    \[
    \begin{array}{ccccccccccc}
    (\widetilde{x}) &   & \cdots    & x_{\ell+6}    & 0 & 0 & 1 & 0 & (x+\widetilde{r})_{\ell+1}    & \cdots    & (x+\widetilde{r})_2       \\
                    & + & \cdots    & r_{\ell+6}    & 0 & 0 & 0 & 0 &   0                           & \cdots    & 0 
    \end{array}
    \]
    Now, using Corollary~\ref{blocage recul}, we obtain \eqref{ccl2 pour bloc de taille 1}.
\end{proof}
The statement of Corollary~\ref{stopping condition bloc 1} means that, given a given block of length $1$ (in the sense that it has one pattern $10$) in $r$, we are able to control the propagation of carries so that the left part of the addition does not change the expansion on the right part and vice versa. We now want to have the same kind of control for larger blocks. We could assume that $x$ has many $0$'s facing the block in $r$ that we want to control. However, this condition would be more and more ``expensive'' (in the sense that the probability for $x$ to satisfy it would decrease to $0$) as the length of the block increases. To avoid this, we are looking for conditions on $x$ that affect only a bounded number of digits, regardless of the length of the block. In other words, we want our conditions to appear in ``most'' of $x\in\XX$. We obtain the following corollary where the length of the block is $m+2$ and where we fixed $8$ digits in $x$.
\begin{corollaire}\label{stopping cdt bloc >1}
Let $x\in\XX$, $r\in\NN$ and $m\ge0$. Suppose that it exists $\ell\ge 2$ such that 
    \begin{itemize}
        \item $x_i=0$ where $i\in\{\ell,\ell+1,\ell+2,\ell+3,\ell+2m+4,\ell+2m+5,\ell+2m+6,\ell+2m+7\}$,
        \item $r_{\ell+1}=r_{\ell+2m+7}=0$ and $r_{\ell+2i+3}=1$ for $i=0,\cdots,m+1$.
    \end{itemize}
    Denote $\widetilde{r}:=\overline{r_{\ell}\cdots r_2}$. Then we have, for all $k\ge\ell+2m+7$
    \begin{align}\label{ccl 1 lemme stopping cdt}
        (x+\widetilde{r}+\sum_{i=0}^{m+1}F_{\ell+2i+3})_k=(x+\widetilde{r})_k=x_k
    \end{align}
    and, for all $k\le \ell+2m+2$
    \begin{align}\label{ccl 2 lemme stopping cdt}
      (x+r)_k=(x+\widetilde{r}+\sum_{i=0}^{m+1}F_{\ell+2i+3})_k.  
    \end{align}
\end{corollaire}
\begin{proof}
For simplicity, we write $B$ as the block $\sum_{i=0}^{m+1}F_{\ell+2i+3}$. We decompose the addition $x+r$ in several steps. First we add $\widetilde{r}$. With the hypothesis on $x$, we actually consider the following addition
\[
\begin{array}{cccccccccccccccccc}
     (x)                &   & \cdots    & x_{\ell+8+2m} & 0 & 0 & 0 & 0 & x_{\ell+2m+3} & \cdots    & x_{\ell+4}    & 0 & 0 & 0 & 0         & x_{\ell-1}    & \cdots    & x_2  \\
     (\widetilde{r})    & + &           &               &   &   &   &   &               &           &               &   &   &   & r_{\ell}  & r_{\ell-1}    & \cdots    & r_2
\end{array}
\]
Thanks to Lemma~\ref{blocage_avancee}, we know that this addition can only modify digits of $x$ of indices $\le \ell+1$. We continue with the addition of $B$. For simplicity, we write $\widetilde{x}=x+\widetilde{r}$
\[
\begin{array}{cccccccccccccccccc}
    (\widetilde{x}) &   & \cdots    & x_{\ell+8+2m} & 0 & 0 & 0 & 0 & x_{\ell+2m+3} & \cdots    & x_{\ell+4}    & 0 & 0 & \widetilde{x}_{\ell+1}  & \widetilde{x}_{\ell}    & \widetilde{x}_{\ell-1}  & \cdots    & \widetilde{x}_2 \\
    (B)             & + &           &               &   &   & 1 & 0 & 1             & \cdots    & 0             & 1 & 0 & 0                         & 0                         & 0                         & \cdots    & 0
\end{array}
\]
Now, Lemma~\ref{blocage avancee + motifs} ensures that the expansion of $\widetilde{x}$ is not modified for digits of indices $\ge \ell+2m+7$. We thus prove \eqref{ccl 1 lemme stopping cdt}. But more precisely, Lemma~\ref{blocage avancee + motifs} concludes that the expansion of $\widetilde{x}+B$ is either
\[\overline{\cdots x_{\ell+8+2m}0100(\widetilde{x}+B)_{\ell+2m+3}\cdots(\widetilde{x}+B)_2 }\]
or
\[\overline{\cdots x_{\ell+8+2m}0010(\widetilde{x}+B)_{\ell+2m+3}\cdots(\widetilde{x}+B)_2 }.\]
In both case, a pattern $00$ appears between the indices $\ell+2m+4$ and $\ell+2m+7$. Thus, we can apply Corollary~\ref{blocage recul}: the final step of the addition, which consists into adding what remains in $r$, will not modify the digits of indices $\le \ell+2m+2$. We obtain the conclusion \eqref{ccl 2 lemme stopping cdt}.
\end{proof}

 

\section{Unique Ergodicity of the Odometer}\label{UEO}

    \subsection{Rokhlin towers and the ergodic measure}\label{RTEM}

In this Subsection, we focus on the action of $T$ on cylinders. For each $k\ge 1$, we consider the partition of $\XX$ into $F_{k+2}$ cylinders corresponding to all possible blocks formed by the rightmost $k$ digits of a Z-adic integers (we call them ``\emph{cylinders of order $k$}'').
For example, at order $1$, we partition $\XX$ into $C_0$ and $C_1$. At order $2$, we get $\XX=C_{00}\sqcup C_{01}\sqcup C_{10}$. At order $3$
\[\XX=C_{000}\sqcup C_{001} \sqcup C_{010} \sqcup C_{100} \sqcup C_{101}.\]
In general, we order lexicographically the cylinders of order $k$:
\begin{itemize}
    \item first, those whose name has a $0$ at the leftmost position (there are $F_{k+1}$ of them),
    \item then those whose name has a $1$ at the leftmost position (there are $F_{k}$ of them).
\end{itemize}
We observe that each cylinders of order $k$ with a $0$ at the leftmost position, except the last one, is mapped by $T$ onto the next one, giving rise to a Rokhlin tower of height $F_{k+1}$ : we call it the \emph{large} tower of order $k$. Similarly, each cylinders or order $k$ with a $1$ at the leftmost position, except the last one, is mapped by $T$ onto the next one, giving rise to another Rokhlin tower of height $F_k$ : we call it the \emph{small} tower of order $k$. Thus, for each $k\ge 1$, we get a partition of $\XX$ into two Rokhlin towers whose levels are all cylinders of order $k$. For instance, for $k=4$, we get 

\begin{figure}[H]
    \centering
    \includegraphics[scale=0.35]{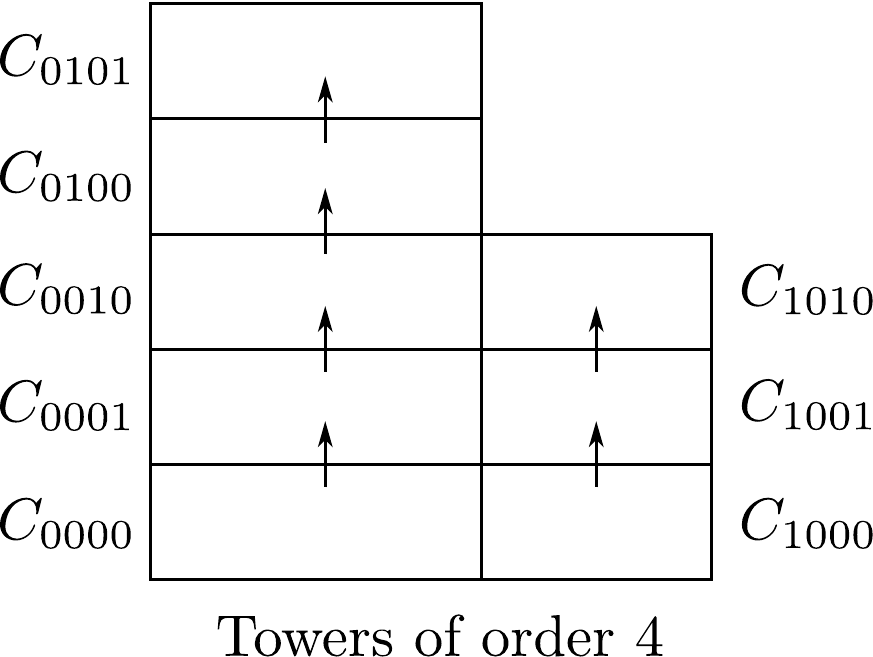}
    \caption{Rokhlin towers of order $4$ (the action of $T$ is representing by the arrows).}
    \label{RT4_Zeck}
\end{figure}
It remains to describe the transition from the Rokhlin towers of order $k$ to those of order $k+1$. First, note that all levels of the small tower of order $k$ are also levels of the large tower of order $k+1$, since a cylinder of order $k$ with a $1$ at the leftmost position coincides with the cylinders of order $k+1$ with an additional $0$ concatenated at the left of its name (e.g. $C_{101000}=C_{0101000}$).

Each level of the large tower of order $k$ is partitioned into two cylinders of order $k+1$ : one obtained by concatenating a $0$ at the left of its name and the other obtained by concatenating a $1$. Thus, the large Rokhlin tower of order $k$ is cut into two subtowers : 
\begin{itemize}
    \item the first one is the bottom part of the large Rokhlin tower of order $k+1$ (the top part being nothing but the small Rokhlin towers of order $k$);
    \item the second one is the small Rokhlin tower of order $k+1$.
\end{itemize}
This transition will be refered as the \emph{Cut and Stack process}.
\begin{figure}[H]
    \centering
    \includegraphics[scale=0.26]{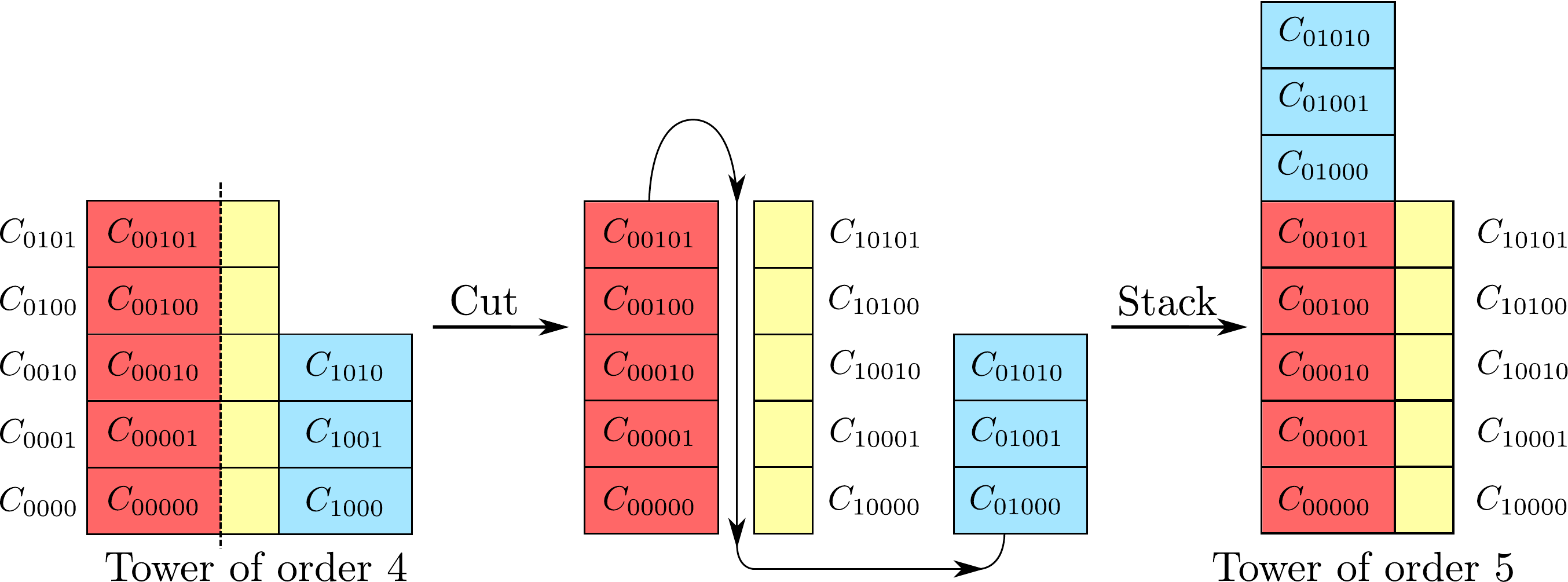}
    \caption{Visual description of the Cut and Stack process (for $k=4$).}
    \label{C_S_Zeck}
\end{figure}

The proceding analysis of the action of $T$ on cylinders, yielding to the construction of the Rokhlin towers, enables us to describe the unique $T$-invariant probability measure.
\begin{prop}\label{prop caract mesure zeck}
$(\XX,T)$ is uniquely ergodic and the unique $T$-invariant measure $\PP$ satisfies $\forall \ell\ge 2$ and $(r_{\ell},\cdots,r_2)\in\XX_f$ 
\begin{align}\label{caract mesure zeck}
    \PP(C_{r_{\ell}\cdots r_2})= &\left\{\begin{array}{cc}
         \dfrac{1}{\varphi^{\ell-1}}& \textrm{ if } r_{\ell}=0 \\
         \dfrac{1}{\varphi^{\ell}}& \textrm{ if } r_{\ell}=1  
    \end{array}\right.
\end{align}
\end{prop}
\begin{proof}
Let $\PP$ be a $T$-invariant measure. Due to $T$-invariance, we observe that, for each order $k\ge1$, the measure of each level of the large tower of order $k$ is the same so as the measure of each level of the small tower of order $k$. We claim that 
\begin{align}\label{mesure C_1}
    \PP(C_1)=\frac{1}{\varphi^2}.
\end{align} 
Indeed, if we denote by $u_k$ (resp. $v_k$) the number of levels included in $C_1$ in the large (resp. small) tower of order $k$, then, due to the Cut and Stack process, we have the relations for $k\ge 2$
\[\left\{\begin{array}{ll}
     u_{k+1}&=u_k+v_k   \\
     v_{k+1}&=u_k 
\end{array}\right.\]
with initial conditions $u_1=0$ and $v_0=1$. We deduce that, for $k\ge 1$, $u_k=F_{k-1}$ and $v_k=F_{k-2}$. Also, for all $k\ge 1$, we have the identity
\begin{align*}
    \PP(C_1)&= u_k~\PP~(x\in \textrm{one level of the large tower of order k})\\
        & \hspace{5mm} + v_k~\PP~(x\in\textrm{one level of the small tower of order k}) \\
        &= \frac{u_k}{F_{k+1}}~\PP~(x\in \textrm{the large tower of order k}) \\
        & \hspace{5mm} + \frac{v_k}{F_k}~\PP~(x\in\textrm{the small tower of order k}) \\
        &=\frac{v_k}{F_k} + (\frac{u_k}{F_{k+1}}-\frac{v_k}{F_k})~\PP~(x\in \textrm{the large tower of order k}) \\
\end{align*}
Taking the limit on $k$ gives \eqref{mesure C_1}. From the towers and the $T$-invariance, we successively deduce that $\PP(C_0)=\frac{1}{\varphi}$, $\PP(C_{00})=\frac{1}{\varphi^2}$, $\PP(C_{10})=\frac{1}{\varphi^3}$. Then, suppose that, at order $k\ge 2$, each level of the large (resp. small) tower measures $\frac{1}{\varphi^k}$ (resp. $\frac{1}{\varphi^{k+1}}$). Observe that a level of the small tower of order $k$ is, due to the Cut and Stack process, exactly a level of the large tower of order $k+1$. Thus, by $T$-invariance, we deduce that each level of the large tower of order $k+1$ measures $\frac{1}{\varphi^{k+1}}$. Also by $T$-invariance, each level of the small tower of order $k+1$ has the same measure. We denote it $p\in [0,1]$. Since there are $F_{k+2}$ (resp. $F_{k+1}$) levels in the large (resp.small) tower of order $k+1$, we have the relation 
\begin{align*}
    \frac{F_{k+2}}{\varphi^{k+1}} + p F_{k+1} &=1
\end{align*}
which is equivalent to the relation
\begin{align*}
    \varphi F_{k+2} + p\varphi^{k+2} F_{k+1} &= \varphi^{k+2}
\end{align*}
Combining with the classical identity $\varphi F_{k+2}+F_{k+1}=\varphi^{k+2}$, we deduce that $p=\frac{1}{\varphi^{k+2}}$. By induction, we prove \eqref{caract mesure zeck}.
\end{proof}

    \subsection{Probabilistic interpretation of the measure}\label{PIM}
    
In the case of the $b$-adic odometer ($b\ge 2)$, the $T$-invariant measure can be interpreted as an independent choice of each digit according to the uniform law on the set of possible digits (see \cite{TREJYH}, page 7). For the Z-adic odometer, it is not as easy to describe. First, the digits do not follow the same law. Indeed, for all $k\ge 2$, $\PP(x_k=1)=\frac{F_{k-1}}{\varphi^{k}}$ and depends on $k$. That gives another proof about the frequency of $1$'s in the Zeckendorf expansion : $\PP(x_k=1)\tend{k}{+\infty}{}{\frac{1}{\varphi^2+1}}$ (see \cite{MG}, \cite{CGL} for other proofs). Therefore, if $\sigma$ is the shift on $\XX$, the law of $x$ is not the same as the law of $\sigma^k(x)$, for $k\ge 1$ which means $\PP$ is not stationary. Furthermore, the choice of a digit is not independent of the other digits because a $1$ must be followed by a $0$. However, the lack of stationarity and independency of $\PP$ is compensed by the following renewal property.
\begin{prop}\label{renouvellement}
   Let $C$ be a cylinder, $k\ge 2$, and $(r_k,\cdots,r_2)\in\XX_f$. Then
\begin{enumerate}
    \item $\PP\left(\sigma^{k}x\in C~\lvert~x\in C_{0r_{k}\cdots r_2}\right)=\PP(C)$ 
\end{enumerate}
and
\begin{enumerate}[resume]
    \item $\PP\left(\sigma^{k+1}x\in C~\lvert~x\in C_{1r_{k}\cdots r_2}\textbf{}\right)=\PP(C)$ (if $r_{k}=0$).
\end{enumerate}
\end{prop}
\begin{proof}
Without a loss of generality, we can suppose $C$ is a cylinder of order $k_0$ with a $0$ at the left side of its name for some $k_0\ge 1$ so $\PP(C)=\frac{1}{\varphi^{k_0}}$. Then
\begin{align*}
    \PP\left(\sigma^{k+2}x\in C~\lvert~x\in C_{0r_{k}\cdots r_2}\right) &=\frac{\PP\left(x\in C_{0r_{k}\cdots r_2}\cap\sigma^{k+2}x\in C\right)}{\PP(x\in C_{0r_{k}\cdots r_2})}.
\end{align*}
We observe that the set $\{x\in C_{0r_{k}\cdots r_2}\cap\sigma^{k+2}x\in C\}$ is actually a cylinder of order $k+k_0$ with a leftmost $0$ in its name. Its measure is therefore $\frac{1}{\varphi^{k+k_0}}$.
\begin{align*}
    \PP\left(\sigma^{k+2}x\in C~\lvert~x\in C_{0r_{k}\cdots r_2}\right) &=\frac{\varphi^{k}}{\varphi^{k+k_0}}=\frac{1}{\varphi^{k_0}}=\PP(C).
\end{align*}
We get the first point of the proposition. Then, we observe that, since $C_{1r_k\cdots r_2}=C_{01r_k\cdots r_2}$, the second is a particular case of the first one (with $k+1$ instead of $k$).
\end{proof}
Once we know the value of $x_k$, the conditionnal law of $x_{k+1}$ depends neither on $k$ or $x_{k-1},\cdots,x_2$. In particular, we get that
\begin{align*}
    \PP(x_{k+2}=1~\lvert~x_{k+1},\cdots,x_2)= &\left\{\begin{array}{cc}
         \dfrac{1}{\varphi^{2}}& \textrm{ if } x_{k+1}=0 \\
         0& \textrm{ otherwise }   
    \end{array}\right.
\end{align*}
and
\begin{align*}
    \PP(x_{k+2}=0~\lvert~x_{k+1},\cdots,x_2)= &\left\{\begin{array}{cc}
         \dfrac{1}{\varphi}& \textrm{ if } x_{k+1}=0 \\
         1& \textrm{ otherwise. }   
    \end{array}\right.
\end{align*}
Therefore, under $\PP$, the digits $x_2$, $x_3$, $\cdots$ form a Markov Chain with transition probabilities given on the figure, starting with the initial law 
\[\PP(x_2=1)=\frac{1}{\varphi^2} \textrm{ and } \PP(x_2=0)=\frac{1}{\varphi}.\]

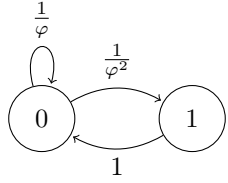
\begin{figure}[H]
    \centering
    \begin{tikzpicture}[shorten >=1pt,node distance=2cm,on grid,auto] 
        \node[state] (q_0)   {$0$}; 
        \node[state](q_3) [right=of q_0] {$1$};
        \path[->] 
        (q_0) edge [bend left] node [] {$\frac{1}{\varphi^2}$} (q_3)
        (q_3) edge [bend left] node [] {1} (q_0)
        (q_0) edge [loop above] node [] {$\frac{1}{\varphi}$} (q_0);
    \end{tikzpicture}
    \caption{Transition probabilities of the Markov Chain.}
    \label{Markov chain}
\end{figure}
    
    \subsection{Reminders on some notions of mixing coefficients}\label{RNMC}
    
In the Introduction (more precisely at Definition~\ref{coeff alpha melange Zeck}), we introduced the notion of $\alpha$-mixing coefficients. It happens that the distribution of the digits in a $Z$-adic integer satisfies a good inequality for another (better) notion of mixing coefficients : the $\phi$-mixing coefficients. They are defined as follows.
\begin{definition}\label{coeff phi melange Zeck}
   Let $(X_j)_{j\ge 1}$ be a (finite or infinite) sequence of random variables. The associated \emph{$\phi$-mixing coefficients} $\phi(k)$, $k\ge 1$, are defined by
   \[\phi(k):=\sup_{p\ge 1} \hspace{1mm} \sup_{A,B} \hspace{2mm} \lvert \PP_A(B)-\PP(B) \rvert\]
   where the second supremum is taken over all events $A$ and $B$ such that 
\begin{itemize}
    \item $A\in\sigma(X_j:1\le  j\le p)$,
    \item $\PP(A)>0$ and
    \item $B\in\sigma(X_j:  j\ge k+p)$.
\end{itemize}
By convention, if $X_j$ is not defined when $j\ge k+p$ then the $\sigma$-algebra is trivial.
\end{definition}
In the case of a finite sequence $(X_1,\cdots,X_n)$, the convention implies that $\phi(k)=0$ for $k\ge n$.

Both $\alpha$ and $\phi$-mixing coefficients are linked together and, as written above, $\phi$-mixing coefficients are ``better'' than $\alpha$-mixing coefficients. Indeed, we have the following property from the survey \cite{RCB1} by Bradley.
\begin{prop}[Bradley]\label{Lien coeff phi et alpha Zeck}
For every $k\ge 1$
\[\alpha(k)\le \frac{1}{2}\phi(k).\]
\end{prop}

    \subsection{\texorpdfstring{$\phi$}{0}-mixing property for Z-adic digits}\label{MPZD}

As mentionned in the previous Subsection, there exists a good upper bound on the $\phi$-mixing coefficients for the coordinates of $x\in\XX$.
\begin{prop}\label{Melange_coord}
For $x=(x_j)_{j\ge 2}\in\XX$ randomly chosen with law $\PP$, we have for $k\ge 1$
\[\phi(k)\le \frac{2}{\varphi^{2k}}.\]
\end{prop}
Our way to prove this result needs to introduce a parametrization of the Markov Chain. So, we define the function 
\[\begin{array}{llll}
    \psi:  & \{0,1\}\times [0,1] & \longrightarrow \{0,1\} \\
         &(0,t) &  \longmapsto 1 &\textrm{ if }t<\frac{1}{\varphi^2}  \\
         &(0,t) & \longmapsto 0 &\textrm{ if }t\ge \frac{1}{\varphi^2} \\
         &(1,t) & \longmapsto 0.
\end{array}\]
Now, let $(U_j)_{j\ge 2}$ independent and identically distributed random variables following a uniform law on the real unit interval and we define the sequence $(y_j)_{j\ge 2}$ as 
\begin{align}\label{genere_coord}
    &\left\{ \begin{array}{lc}
    y_j=\psi(y_{j-1},U_j) & \forall j\ge 2 \\
    y_1=0. & 
\end{array}\right.
\end{align}
\begin{lemme}\label{Automate=Loi}
    The law of $(y_j)_{j\ge 2}$ is $\PP$.
\end{lemme}
\begin{proof}[Proof of Lemma~\ref{Automate=Loi}]
The transition matrix is the same as for $\PP$ and so is the initial law : the probability that $y_2$ is $0$ is the same as the event $U_2\ge \frac{1}{\varphi^2}$ that is $\frac{1}{\varphi}=\PP(C_0)$.
\end{proof}

\begin{proof}[Proof of Proposition~\ref{Melange_coord}]
Let $k\ge 1$ and $p\ge 2$. To choose $x$ with law $\PP$ is equivalent to construct a sequence $y=(y_j)_{j\ge 2}$ using the process $\eqref{genere_coord}$ thanks to Lemma~\ref{Automate=Loi}. Thus, we have a sequence of independent and identically distributed random variables $(U_j)_{j\ge 2}$ with a uniform law on the real unit interval. We consider the event 
\[C:=\left\{\exists j\in]p,p+k] : U_j\ge \frac{1}{\varphi^2}\right\}\]
which has a probability 
\[\PP(C)=1-\frac{1}{\varphi^{2k}}>0.\]
If $y\in C$ then it implies that $y_j=0$ because of $\eqref{genere_coord}$. Then, we take $A\in\sigma(x_j:2\le  j\le p)$ and $B\in\sigma(x_j:  j\ge k+p)$. We observe that, due to \eqref{genere_coord}, $A\in\sigma(U_j : 2\le j \le p)$ while $C\in\sigma(U_j: p<j\le k+p)$ thus $A$ and $C$ are independent. Also, due to Proposition~\ref{renouvellement}, we observe $B\cap C\in\sigma(U_j:j>p)$. So $A$ and $B\cap C$ are independent. 
Also, we have
\begin{equation*}
    \begin{split}
        \lvert\PP_A(B)-\PP(B)\rvert & \le\lvert\PP_A(B)-\PP_{A\cap C}(B)\rvert +\lvert\PP_{A\cap C}(B)-\PP_C(B)\rvert\\
        & \hspace{1 cm} + \lvert\PP_C(B)-\PP(B)\vert.        
    \end{split}
\end{equation*}
But the independences between $A$ and $C$ or $A$ and $B\cap C$ give that 
\begin{equation}\label{ineg triang Zeck}
    \begin{split}
        \lvert\PP_A(B)-\PP(B)\rvert & \le\lvert\PP_A(B)-\PP_{A\cap C}(B)\rvert + \lvert\PP_C(B)-\PP(B)\vert.         
    \end{split}
\end{equation}
As shown in \cite{TREJYH} (Lemma 4.3), denoting $\overline{C}$ the complement of $C$, we have the general inequality for any event $D$
\begin{align}\label{ineg generale conditionnement Zeck}
    \lvert\PP_C(D)-\PP(D)\rvert &\le \PP(\overline{C})
\end{align}
which implies in \eqref{ineg triang Zeck} that
\[\lvert\PP_A(B)-\PP(B)\rvert\le 2\PP(\overline{C}).\]
\end{proof}
Another property about the measure is that if indeed $\PP(x_k=1)$ depends on $k$, it is actually bounded between two positive values. We generalise that fact with the next proposition.
\begin{prop}\label{borne proba Zeck}
Let $(k_0,\cdots,k_{\ell})$ be a collection of integers such that 
\[2\le k_0<k_1<\cdots<k_{\ell}\]
and also $A$ be a union of cylinders of order $k_0-1$ such that $x_{k_0}=0$ if $x\in A$ and such that $\PP(A)>0$. Then
\[\frac{1}{\varphi^{\ell}}\le \PP_A~(\forall 1\le i\le \ell~:~x_{k_i}=0)\le \left(\frac{2}{\varphi^2}\right)^{\ell}.\]
\end{prop}
\begin{proof}
Let $I\subset \NN$. Since there is a finite number of cylinders of order $k_0-1$, we can write $A$ as a disjoint union $\sqcup_{i\in I} C^{(i)}$ where $C^{(i)}$ is a cylinder of order $k_0-1$. We have the identity
\begin{align}\label{etap prop borne proba}
    \PP_A~(\forall 1\le i\le \ell~:~x_{k_i}=0)&=\prod_{i=1}^{\ell}\PP~(x_{k_i}=0~\lvert~A\cap \bigcap_{j=1}^{i-1}(x_{k_j}=0))
\end{align}
with the convention $\cap_{j=1}^0 (x_{k_j}=0)=\XX$.
But we have 
\begin{align*}
    \PP_A(x_{k_1}=0)&=\frac{\PP\left(\left(x_{k_1}=0\right)\cap \left(x\in\sqcup_{i\in I} C^{(i)}\right)\right)}{\PP(A)}\\
        &=\frac{1}{\PP(A)}\sum_{i\in I} \PP\left((x_{k_1}=0)\cap x\in C^{(i)}\right).
\end{align*}
We recall $\PP(A)=\frac{\lvert I\rvert}{\varphi^{k_0-1}}$ and  that the summand is the probability of a disjoint union of $F_{k_1-k_0+1}$ cylinders of order $k_1-1$ such that $x_{k_1}=0$ so its measure is $\frac{F_{k_1-k_0+1}}{\varphi^{k_1-1}}$. Thus, we obtain
\[\PP_A(x_{k_1}=0)=\PP(x_{k_1-k_0+1}=0).\]
We proceed similarly for the other terms in \eqref{etap prop borne proba} and get
\begin{align}
    \PP_A~(\forall 1\le i\le \ell~:~x_{k_i}=0)&=\prod_{i=1}^{\ell}\PP~(x_{k_i-k_{i-1}+1}=0).
\end{align}
The last equality is given using the renewal of $\PP$. Now we claim that, for any $k\ge 2$
\[\frac{1}{\varphi}\le \PP(x_k=0)\le \frac{2}{\varphi^2}.\]
Indeed, we recall $\PP(x_k=0)=\frac{F_k}{\varphi^{k-1}}$ and observe that the subsequences $\left(\frac{F_{2k}}{\varphi^{2k-1}}\right)$ and $\left(\frac{F_{2k+1}}{\varphi^{2k}}\right)$ are adjacent sequences.
\end{proof}

    \subsection{Ergodic convergence}

This subsection is exactly the same as the end of Subsection 2.1 in \cite{TREJYH} in the context of an integer base but we write it again in the Z-adic context where the arguments are identical.

For $x$ in $\XX$, we define the sequence of empirical probability measures along the (beginning of the) orbit of $x$: for every $N\ge 1$, we set
\[\epsilon_N(x):=\frac{1}{N}\sum_{0\le n<N}\delta_{T^nx}\] 
(where $\delta_y$ denotes the Dirac measure on $y\in\XX$).\\
Since the space of probability measures on $\XX$ is compact for the weak-$*$ topology, we can always extract a convergent subsequence. Moreover, every subsequential limit of $(\epsilon_N(x))$ is a $T$-invariant probability measure. By the uniqueness of the $T$-invariant probability measure, for every $x\in \XX$ we have $\epsilon_N(x)\rightarrow \PP$. In other words, we have the convergence 
\begin{equation}\label{uniq_ergo+birk1 Zeck}
        \forall x\in \XX, \hspace{2mm} \forall f\in\mathcal{C}(\XX), \hspace{5mm} \frac{1}{N}\sum_{0\le n< N}f(T^nx) \tend{N}{+\infty}{}{\int_{\XX} f \mathrm{d}\PP}.
\end{equation}
We will be interested here in the special case $x=0$ because $\NN=\{T^n0:n\in\NN\}$. Then \eqref{uniq_ergo+birk1 Zeck} becomes
\begin{equation}\label{uniq_ergo+birk Zeck} 
        \forall f\in\mathcal{C}(\XX), \hspace{5mm} \frac{1}{N}\sum_{0\le n< N}f(n) \tend{N}{+\infty}{}{\int_{\XX} f \mathrm{d}\PP}.
\end{equation}
Equation \eqref{uniq_ergo+birk Zeck} shows that, for a continuous function $f$, averaging $f$ over $\NN$ (for the natural density) amounts to averaging over $\XX$ (for $\PP$). The next section shows how this convergence can be extended to some non-continuous functions related to the sum-of-digits function.

\section{Sum of digits on the Odometer}\label{SOzeck}

This section is very similar to the corresponding section in \cite{TREJYH} for the integer base case. The only differences are the need of the new Lemma~\ref{On Delta1 Zeck}. All the other results in this section are the same even though we need to adapt the proofs.
For every integer $k\ge 2$,  we define the continuous map $s_k:\XX\rightarrow\ZZ$ as the sum of the digits of indices $\le k$, that is to say 
\begin{align*}
    s_k(x):=x_k+\cdots +x_2.
\end{align*}
Let $r\in\NN$. We define the functions $\Delta^{(r)}_k:\XX\rightarrow\ZZ$ by 
\begin{align*}
    \Delta_k^{(r)}(x):=s_k(x+r)-s_k(x).
\end{align*}
The functions $\Delta_k^{(r)}$ are well-defined, continuous (and bounded) on $\XX$. By \eqref{uniq_ergo+birk Zeck}, we have 
\begin{align}\label{birkhoff_result_delta_k Zeck}
         \frac{1}{N}\sum_{n< N} \Delta_k^{(r)}(n)&=\frac{1}{N}\sum_{n< N} \Delta_k^{(r)}(T^n0)\tend{N}{+\infty}{}{\int_{\XX} \Delta_k^{(r)} \mathrm{d}\PP}.
\end{align}
Although the sum-of-digits function $s$ is not well defined on $\XX$, we can extend the function $\Delta^{(r)}$ defined by \eqref{def delta zeck} on the set of $x\in\XX$ for which the number of different digits between $x$ and $x+r$ is finite. This subset contains the $Z$-adic integers $x$ such that there exists an index $k\ge 2+\max(\{\ell: r_{\ell}\neq 0\})$ such that $x_{k}=x_{k+1}=0$ (see Lemma \ref{blocage_avancee}). So, except for a finite number of $Z$-adic integers, we can define 
\[\Delta^{(r)}(x):=\lim\limits_{k\to\infty}\Delta^{(r)}_k(x).\]
\begin{remarque}
Let $t,u$ be two integers. For every integer $k$ we have the decomposition formula 
\begin{align}\label{decomp_sumk Zeck}
    \Delta_k^{(t+u)}&=\Delta_k^{(t)}+\Delta_k^{(u)}\circ T^t.
\end{align}
So, taking $\PP$-almost everywhere the limit when $k$ tends to infinity, we get
\begin{align}\label{decomp_sum Zeck}
    \Delta^{(t+u)}&=\Delta^{(t)}+\Delta^{(u)}\circ T^t \hspace{2mm}  (\PP\textrm{-a-s.}).
\end{align}
Then, by induction on $t$, we deduce
\begin{align}\label{decomp_gene Zeck}
    \Delta^{(t)}&=\Delta^{(1)}+\Delta^{(1)}\circ T+\cdots+\Delta^{(1)}\circ T^{t-1} \hspace{2mm}  (\PP\textrm{-a-s.}).
\end{align}
\end{remarque}
$\Delta^{(r)}$ is not bounded on $\XX$ therefore it is not continuous. So, \eqref{uniq_ergo+birk Zeck} is not applicable for $f\circ \Delta^{(r)}$ with $f$ a continuous map on $\XX$. However, it is possible to get the same convergence as in \eqref{uniq_ergo+birk Zeck} with weaker assumptions on $f$ than continuity.
\begin{prop}\label{prop_moment Zeck}
Let $r\ge 1$ and $f:\ZZ\rightarrow\CC$. Assume that there exist  $\alpha\ge 1$ and $C$ in $\RR^*_+$  such that for every $n\in\ZZ$  
\begin{align}\label{hyp_prop_moment2 Zeck}
    \lvert f(n) \rvert &\le C\lvert n \rvert^{\alpha} +\lvert f(0)\rvert.
\end{align}
Then $f\circ \Delta^{(r)}\in L^1\left(\PP\right)$ and we have the convergence
\begin{align*}
    \lim\limits_{N\to\infty}\frac{1}{N}\sum_{n<N}f(\Delta^{(r)}(n))&=\int_{\XX} f(\Delta^{(r)}(x)) \mathrm{d}\PP(x)\\
            &= \lim\limits_{k\to\infty} \int_{\XX} f(\Delta^{(r)}_k(x)) \mathrm{d}\PP(x).
\end{align*}
\end{prop}

\begin{corollaire}\label{corollaire_prop_moment Zeck}
 For every $d\in\ZZ$
\begin{align}\label{lien_mu^r_et_mu Zeck}
        \mu^{(r)}(d)&:=\lim\limits_{N\to\infty}\frac{1}{N} \left|\left\{n<N: \Delta^{(r)}(n)=d\right\}\right| \notag \\
        &=\PP\left(\left\{x\in \XX: \Delta^{(r)}(x)=d\right\}\right).
\end{align}
Moreover, $\Delta^{(r)}$ has zero-mean and has finite moments. 
\end{corollaire}
The proof of this corollary is exactly the same as in \cite{TREJYH}. In particular, we just prove the existence of the asymptotic densities of the sets $\{n\in\NN~\lvert~ \Delta^{(r)}(n)=d\}$, where $d\in\ZZ$. In integer base, it is easy to prove (see \cite{JB} or \cite{TREJYH}). It is harder here to follow Besineau's proof because it uses the arithmetic properties of the sum-of-digits  (in integer base) function that we do not have anymore. 
\begin{remarque}\label{et pour 0 ? Zeck}
    Using trivial arguments, Proposition~\ref{prop_moment Zeck} and Corollary~\ref{corollaire_prop_moment Zeck} are also true when $r=0$. We observe that $\mu^{(0)}=\delta_0$.
\end{remarque}
Before proving this proposition and its corollary, we need the following lemma that looks like Lemma~1.29 in \cite{LS}.
\begin{lemme}\label{lemme_technique Zeck}
    Let $r\ge 1$. For $N\ge 1$, $k\ge 2$ and $d,d'\in\ZZ$, we have the inequality
    \begin{equation}\label{ineg_techni2 Zeck}
        \begin{split}
             \frac{1}{N}\left| \{n<N: (\Delta^{(r)}(n),\Delta^{(r)}_k(n))=(d,d')\} \right| & \\
             & \hspace{-3 cm}\le r\varphi^3  \hspace{1mm}\PP \hspace{-1mm} \left( \{x\in \XX: (\Delta^{(r)}(x),\Delta^{(r)}_k(x))=(d,d')\}\right).
        \end{split}
    \end{equation}
    In particular, we have
    \begin{align}\label{ineg_techni Zeck}
        \frac{1}{N}\left| \{n<N: \Delta^{(r)}(n)=d\} \right| \le r\varphi^3 \hspace{1mm} \PP \hspace{-1mm} \left( \{x\in \XX: \Delta^{(r)}(x)=d\}\right).
    \end{align}
\end{lemme}
\begin{proof}[Proof of Lemma~\ref{lemme_technique Zeck}]
    We adapt the proof of Lemma~2.3 in \cite{TREJYH}. First, \eqref{ineg_techni2 Zeck} implies \eqref{ineg_techni Zeck} so we just prove \eqref{ineg_techni2 Zeck}. We fix $k\ge 2$. For $\ell\ge 1$, let $V_{\ell}$ be the set of the values reached by the couple $(\Delta^{(r)},\Delta^{(r)}_k)$ on the $F_{\ell+1}-r$ (resp. $F_{\ell}-r$) first levels of the large (resp. small) tower of order $\ell$. Of course, if $F_{\ell+1}-r\le 0$ then $V_{\ell}:=\emptyset$. In particular, for any $r\ge 1$, $V_1=\emptyset$. Also, if $F_{\ell}\le r<F_{\ell+1}$ then $V_{\ell}$ is defined considering only the large tower of order $\ell$. For any $\ell$, we observe that $V_{\ell}$ is a finite set. On each level that is not in the $r$ top levels of the large or small tower, the first $\ell$ digits of both $x$ and $x+r$ are constant, and digits of higher order are the same. Therefore, $\Delta^{(r)}$ and $\Delta^{(r)}_k$ are constant on such a level. We observe that the sequence $(V_{\ell})_{\ell\ge 1}$ is increasing for the inclusion.
    
    Now, for $d,d'\in\ZZ$, there are $2 $ cases.
    \begin{enumerate}
        \item If $(d,d')\notin\cup_{\ell\ge 2} V_{\ell}$, then for each $n\in\NN$ we have $(\Delta^{(r)}(n),\Delta^{(r)}_k(n))\neq(d,d')$. Indeed, for each $n\in\NN$, there exists a smallest integer $\ell\ge 2$ such that $n$ is in the first $F_{\ell+1}-r$ levels of the large tower of order $\ell$, hence $(\Delta^{(r)}(n),\Delta^{(r)}_k(n)) \in V_{\ell}$. In this case, \eqref{ineg_techni2 Zeck} is trivial.
        \item If $(d,d')\in\cup_{\ell\ge 1} V_{\ell}$ then there exists a unique $\ell\ge 2$ such that $(d,d')\in V_{\ell}\backslash V_{\ell-1}$. We observe that, due to the Cut and Stack process, the value $(d,d')$ must appear firstly in the large tower of order $\ell$. Since $(\Delta^{(r)},\Delta^{(r)}_k)$ is constant on each of the first $F_{\ell+1}-r$ levels of the large tower, it takes the value $(d,d')$ on at least one whole such level which is of measure $\frac{1}{\varphi^{\ell}}$. So, we have  
        \begin{align}\label{etap1 lemme tech zeck}
            \PP\hspace{-1mm} \left( \{x\in \XX: (\Delta^{(r)}(x),\Delta^{(r)}_k(x))=(d,d')\}\right)\ge \frac{1}{\varphi^{\ell}}.
        \end{align}
    Also, since the couple $(d,d')\notin V_{\ell-1}$, we claim that, for every $N\ge 1$
    \[\frac{1}{N}\left| \{n<N: (\Delta^{(r)}(n),\Delta^{(r)}_k(n))=(d,d')\} \right|\le \frac{r}{F_{\ell-1}}.\]
    Indeed, 
    \begin{enumerate}
        \item If $r\ge F_{\ell-1}$, the inequality is then trivial.
        \item If $r< F_{\ell-1}$ then, since $(d,d')$ is not in $V_{\ell-1}$, $(d,d')$ can only appear inside a part the  $r$ highest levels of the big or the small towers of order $\ell-1$. Let us denote $C$ the union of these $r$ highest levels. Since $0$ lies in the bottom level of the large tower of order $\ell-1$, the set $S$ of integers $n\ge 0$ such that $T^n0\in C$ has the following properties :
        \begin{itemize}
            \item $\{0,\cdots,F_{\ell}-r-1\}\cap S=\emptyset$ and
            \item $S$ is the union of subsets formed by $r$ consecutive integers separated by gaps of length $F_{\ell}-r$ or $F_{\ell-1}-r$.
        \end{itemize}
        So, we have
        \begin{align*}
            \frac{1}{N}\left| \{0\le n<N: (\Delta^{(r)}(n),\Delta^{(r)}_k(n))=(d,d')\} \right|&\\
            &\hspace{-1.38cm}\le \frac{1}{N}\left| \{0\le n<N: T^n0\in C \}\right|\\
            &\hspace{-1.38cm}\le\frac{r}{F_{\ell-1}}.
        \end{align*}
    \end{enumerate}
    Since $F_{\ell-1}\ge \varphi^{\ell-3}$ (by double induction), we get
    \begin{align}\label{etap2 lemme tech zeck}
        \frac{1}{N}\left| \{n<N: (\Delta^{(r)}(n),\Delta^{(r)}_k(n))=(d,d')\} \right|\le \frac{r}{\varphi^{\ell-3}}.
    \end{align}
    Combining inequalities \eqref{etap1 lemme tech zeck} and \eqref{etap2 lemme tech zeck} gives \eqref{ineg_techni2 Zeck}.
    \end{enumerate}
    \begin{figure}[H]
        \centering
        \includegraphics[scale=0.33]{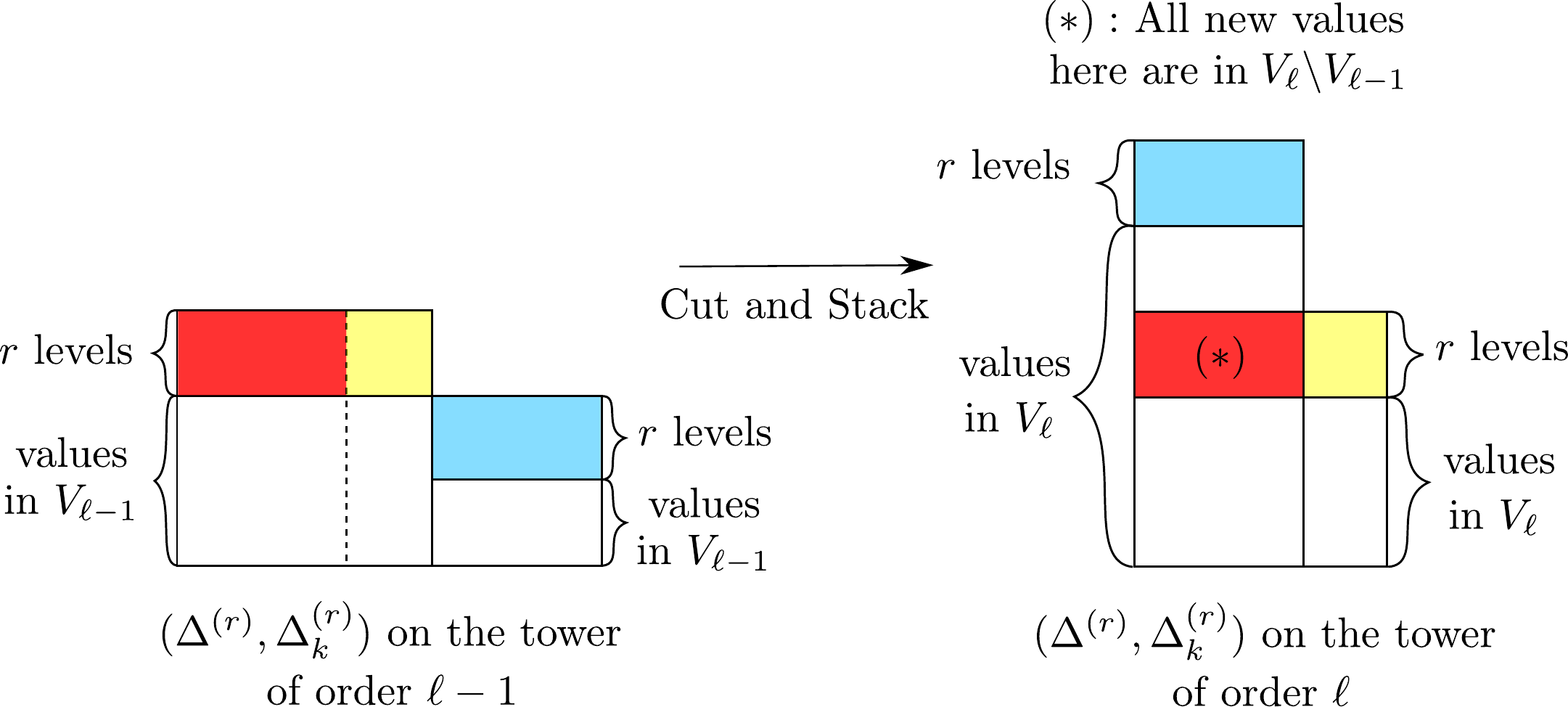}
        \caption{Visual description of $V_{\ell-1}$, $V_{\ell}$ and $V_{\ell}\backslash{V_{\ell-1}}$.}
        \label{Vl Zeck}
    \end{figure}
\end{proof}
As shown in \eqref{decomp_gene Zeck}, the understanding of $\Delta^{(1)}$ (and so $\Delta_k^{(1)}$ for $k\ge 2$) is fundamental to understand $\Delta^{(r)}$ so we also state the following lemma.
\begin{lemme}\label{On Delta1 Zeck}
    The function $\Delta^{(1)}$ is well-defined for $x\in\XX$ if and only if $x\in C_{00(10)^{d}}\cup C_{001(01)^d}$ for some $d\ge 0$. 
    
    Furthermore, if $x\in C_{00(10)^{d}}$ then $\Delta^{(1)}(x)=1-d$ and 
    \[\Delta^{(1)}_k(x)=\left\{\begin{array}{ll}
        \frac{2-k}{2} & \textrm{ if } k\equiv 0 [2] \textrm{ and } k\le 2d,  \\
        \frac{1-k}{2} & \textrm{ if } k\equiv 1 [2] \textrm{ and } k\le 2d+1,  \\
        1-d & \textrm{ if } k\ge 2d+2.  \\
    \end{array}\right.\]
    
    Also, if $x\in C_{001(01)^{d}}$ then $\Delta^{(1)}(x)=-d$ and 
    \[\Delta^{(1)}_k(x)=\left\{\begin{array}{ll}
        \frac{-k}{2} & \textrm{ if } k\equiv 0 [2] \textrm{ and } k\le 2d+2, \\
        \frac{1-k}{2} & \textrm{ if } k\equiv 1 [2] \textrm{ and } k\le 2d+1,  \\
        \frac{3-k}{2} & \textrm{ if } k=2d+3,\\
        -d & \textrm{ if } k\ge 2d+4.  \\
    \end{array}\right.\]
\end{lemme}
\begin{proof}[Proof of Lemma \ref{On Delta1 Zeck}]
    If $x\in C_{00(10)^{d}}$ for some $d\ge 0$, the definition of $T$ gives that $x+1\in C_{01(00)^{d}}$ with the other digits unchanged. Thus the sequence $(\Delta^{(1)}_k(x))$ is stationary, $\Delta^{(1)}(x)$ is well defined and $\Delta^{(1)}(x)=1-d$. Also, if $x\in C_{001(01)^{d}}$ for some $d\ge 0$, then $x+1\in C_{010(00)^{d}}$, $\Delta^{(1)}(x)$ is well defined and $\Delta^{(1)}(x)=-d$. If $x$ does not belong to those cylinders then $x=(01)^{\infty}$ or  $x=(10)^{\infty}$. In that case, $x+1=\overline{0^{\infty}}$ and the sequence $(\Delta^{(1)}_k(x))$ diverges to $-\infty$. This proves the equivalence. The values of $\Delta_k^{(1)}(x)$ can be easily found by looking at the addition we write in \ref{HTAI}.
\end{proof}
\begin{proof}[Proof of Proposition \ref{prop_moment Zeck}]
We adapt the proof of Proposition 2.1 in \cite{TREJYH}. Let $\varepsilon>0$. For any integer $k\ge 2$, we have
\begin{align*}
    \left\vert\frac{1}{N}\sum_{n<N} f(\Delta^{(r)}(n))-\int_{\XX} f(\Delta^{(r)}(x)) \mathrm{d}\PP(x)\right\vert 
        &\le A_1+A_2+A_3,
\end{align*}
where
\begin{align*}
    A_1 &:= \displaystyle\frac{1}{N}\sum_{n<N} \left\vert f(\Delta^{(r)}(n))-f(\Delta^{(r)}_k(n))\right\vert, \\ 
    A_2 &:= \displaystyle\left\vert\frac{1}{N}\sum_{n<N}f(\Delta^{(r)}_k(n))-\int_{\XX} f(\Delta^{(r)}_k(x)) \mathrm{d}\PP(x)\right\vert, \\
    A_3 &:=\displaystyle\int_{\XX} \left\vert f(\Delta^{(r)}(x))-f(\Delta^{(r)}_k(x))\right\vert \mathrm{d}\PP(x).
\end{align*}
However, as in \cite{TREJYH}, we have
\begin{align*}
    A_1&=\frac{1}{N}\sum_{n<N} \sum_{j,j'\in\ZZ}\Big\vert f(j)-f(j')\Big\vert\mathbbm{1}_{(j,j')}\left(\Delta^{(r)}(n),\Delta^{(r)}_k(n)\right)\\
    &= \sum_{j,j'\in\ZZ}\Big\vert f(j)-f(j')\Big\vert\frac{1}{N}\sum_{n<N}\mathbbm{1}_{(j,j')}\left(\Delta^{(r)}(n),\Delta^{(r)}_k(n)\right)\\
    &\le r \varphi^3\sum_{j,j'\in\ZZ}\Big\vert f(j)-f(j')\Big\vert \hspace{1mm}\PP \hspace{-1mm} \left( \Delta^{(r)}(n) = j , \Delta^{(r)}_k(n)=j'\right)\\
    &=r \varphi^3\int_{\XX} \Big\vert f(\Delta^{(r)}(x))-f(\Delta^{(r)}_k(x))\Big\vert\mathrm{d}\PP(x)=r \varphi^3 A_3.
\end{align*}
It follows that
\begin{align*}
    \left\vert\frac{1}{N}\sum_{n<N} f(\Delta^{(r)}(n))-\int_{\XX} f(\Delta^{(r)}(x)) \mathrm{d}\PP(x)\right\vert 
        &\le A_2 + (1+r \varphi^3)A_3.
\end{align*}
We want to apply a Dominated Convergence Theorem to deal with $A_3$ (observe that we have the convergence $f(\Delta^{(r)}_k(x))\tend{k}{\infty}{}{f(\Delta^{(r)}(x))}$ $\PP$-almost-surely). For this, we need to find a good dominant function. As in \cite{TREJYH}, we define $g_i:=\sup_{k\ge 2}|\Delta^{(1)}_k\circ T^i(x)|$ for $i=0,\cdots,r-1$ and by \eqref{hyp_prop_moment2 Zeck} we get the inequalities
\[\left\lvert f\circ\Delta^{(r)}_k(x)\right\rvert\le C \sum_{j_0+\cdots+j_{r-1}=\alpha}\sum_{i=0}^{r-1}\frac{\binom{\alpha}{j_0,\cdots,j_{r-1}}}{r}g_i(x)^{rj_i}+\Big\lvert f(0)\Big\rvert\]
and
\[\left\lvert f\circ\Delta^{(r)}(x)\right\rvert\le C \sum_{j_0+\cdots+j_{r-1}=\alpha}\sum_{i=0}^{r-1}\frac{\binom{\alpha}{j_0,\cdots,j_{r-1}}}{r}g_i(x)^{rj_i}+\Big\lvert f(0)\Big\rvert.\]
We need to prove that $g_i^{rj_i}$ is integrable for the measure $\PP$. It is equivalent to show that $\sum_m\PP(\{x\in\XX : g_i^{rj_i}(x)>m\})$ is a convergent series. We have 
\begin{align*}
    g_i(x)> m^{\frac{1}{rj_i}} &\Leftrightarrow \sup_{k\ge 2}|\Delta^{(1)}_k\circ T^i(x)|> m^{\frac{1}{rj_i}}\\&\Leftrightarrow \exists k\ge 2, \hspace{2mm} \lvert\Delta_k^{(1)}(T^ix) \rvert > m^{\frac{1}{rj_i}}. \\
\end{align*}
From Lemma~\ref{On Delta1 Zeck}
\begin{align*}
    \exists k\ge 2, \hspace{2mm} \lvert\Delta_k^{(1)}(T^ix) \rvert > m^{\frac{1}{rj_i}} &\Leftrightarrow T^ix\in C_{00(10)^{\lfloor m^{\frac{1}{rj_i}}\rfloor}}\cup C_{0(01)^{\lfloor m^{\frac{1}{rj_i}}\rfloor}}
\end{align*}
It follows, by $T$-invariance
\[\PP(\{x\in\XX : g_i^{rj_i}(x)>m\})\le \left(\frac{1}{\varphi}\right)^{2+2\lfloor m^{\frac{1}{rj_i}}\rfloor}+\left(\frac{1}{\varphi}\right)^{1+2\lfloor m^{\frac{1}{rj_i}}\rfloor}=\left(\frac{1}{\varphi}\right)^{2\lfloor m^{\frac{1}{rj_i}}\rfloor}.\]
The quantity on the RHS is the general term of a convergent series which shows that $g_i^{rj_i}$ is integrable for the measure $\PP$. The dominated theorem can be applied and, for $k$ large enough, $(1+r \varphi^3)A_3\le \frac{\varepsilon}{2}$ for every $N\ge 1$. Now, once we have fixed such a $k$, for $N$ large enough, $A_2$ is bounded by $\frac{\varepsilon}{2}$ because of \eqref{uniq_ergo+birk Zeck} and the continuity of $\Delta_k^{(r)}$ and $f$. The convergence in the statement is thus proved.
 
 Note that the argument of the dominated convergence theorem also proves that $f\circ\Delta^{(r)}\in L^1(\PP)$ and $\int_{\XX} f\circ \Delta^{(r)}\mathrm{d}\PP=\lim\limits_{k\rightarrow \infty}\int_{\XX} f\circ\Delta^{(r)}_k\mathrm{d}\PP$.
\end{proof}

More generally, we have the following convergence.
\begin{prop}\label{prop_gene Zeck}
Let $r\ge 1$ and $f:\ZZ\rightarrow\CC$ be such that $f\circ\Delta^{(r)}\in L^1(\PP)$. Then 
    \[\lim\limits_{N\to\infty}\frac{1}{N}\sum_{n<N}f(\Delta^{(r)}(n))=\int_{\XX} f(\Delta^{(r)}(x)) \mathrm{d}\PP(x).\]
\end{prop}
The proof is exactly the same as Proposition 2.4 in \cite{TREJYH}. We have shown that the random variable $\Delta^{(r)}$ satisfies some good properties such as the finiteness of moments of any order of $\mu^{(r)}$, the law of $\Delta^{(r)}$. The next section focuses on another natural question : how to compute the law $\mu^{(r)}$?

\section{How to compute \texorpdfstring{$\mu^{(r)}$}{yr}}\label{HTCM}

In this section, we present an algorithm that computes exactly the measure $\mu^{(r)}$ for any $r$ in $\NN$ and the consequences of it. 

    \subsection{Description of the algorithm}\label{DOTA}
    
This algorithm is different from the classical way to compute $\mu^{(r)}$ in integer base (see \cite{JB}, page 14 or \cite{TREJYH} Proposition 3.1 for instance). In integer base $b\ge 2$, the computation of $\mu^{(r)}$ relies on inductive relations on the expansion of $r$ which are easy to prove in that case. For instance, one of the relation states that if $r$ is a multiple of $b$ then the variation when adding $r$ to $x$ is the same as the variation when adding $\frac{r}{b}$ to $\sigma x$. So the laws of the variation when adding $r$ or $\frac{r}{b}$ are the same. In the Zeckendorf representation, this trivial argument is false because the carries can modify the digits on both sides (unlike the integer base case where only the left side can be affected). However, we find an algorithm to compute $\mu^{(r)}$, in this Zeckendorf system, that relies on Rokhlin towers. We precise that this algorithm can be adapted in integer base.

First, let $r\ge 1$ ($r=0$ is irrelevant) and let $\ell\ge 2$ be the unique integer such that $F_{\ell}\le r< F_{\ell+1}$. At a given order $k\ge 1$, we define $\CST{k}$ as the union of the levels of the large and small towers of order $k$, except the $r$ highest levels. We observe $\CST{1}=\emptyset$. On each level of $\CST{k}$, $\Delta^{(r)}$ is constant (see Proof of Lemma \ref{lemme_technique Zeck}). But, due to the Cut-and-Stack process, the value of $\Delta^{(r)}$ on these ``constant'' levels of $\CST{k}$ may be deduced from those taken in $\CST{k-1}$. That is why we also introduce, for $k\ge 2$, the \emph{New information zone of the Rokhlin tower of order $k$} as the set 
\[\NIZ{k}:=\CST{k}\backslash{\CST{k-1}}.\]

Of course, everything depends on $r$ but we do not emphasise on that in the notations of $\ell$, $\NIZ{k}$ and $\CST{k}$ for simplicity. Let us define clearly this set $\NIZ{k}$.
\begin{prop}\label{Def NIZ_k}
    In a large tower, we enumerate, starting by $1$, the levels from the base of the tower to the top of it. We have the following description of $\NIZ{k}$.
    \begin{itemize}
        \item If $2\le k\le \ell-1$, $\NIZ{k}=\emptyset$.
        \item $\NIZ{\ell}$ is the union of the $F_{\ell+1}-r$ first levels of the large tower of order $\ell$.
        \item $\NIZ{\ell+1}$ is the union of the $F_{\ell}$ levels between the $F_{\ell+1}-r+1^{th}$ and the $F_{\ell+2}-r^{th}$ levels of the large tower of order $\ell+1$.
        \item If $k>\ell+1$, $\NIZ{k}$ is the union of the $r$ levels between the $F_{k}-r+1^{th}$ and the $F_k^{th}$ levels of the large tower of order $k$.
    \end{itemize}
\end{prop}
\begin{remarque}
    By definition, for every $k\ge 1$, the set $\NIZ{k}$ is either empty or a disjoint union of cylinders of order $k$ (exactly $F_{\ell+1}-r$ if $k=\ell$, $F_{\ell}$ if $k=\ell+1$ or $r$ if $k\ge\ell+2$). Moreover, we observe that $(\NIZ{k})_{k\ge \ell}$ is a partition of $\XX\backslash{\{(01)^{\infty}, (10)^{\infty}\}}$.
\end{remarque}

\begin{exemple}\label{Example NIZ pour r=4}Here, an example with $r=4=\overline{101}$.
\begin{figure}[H]
    \centering
    \includegraphics[scale=0.34]{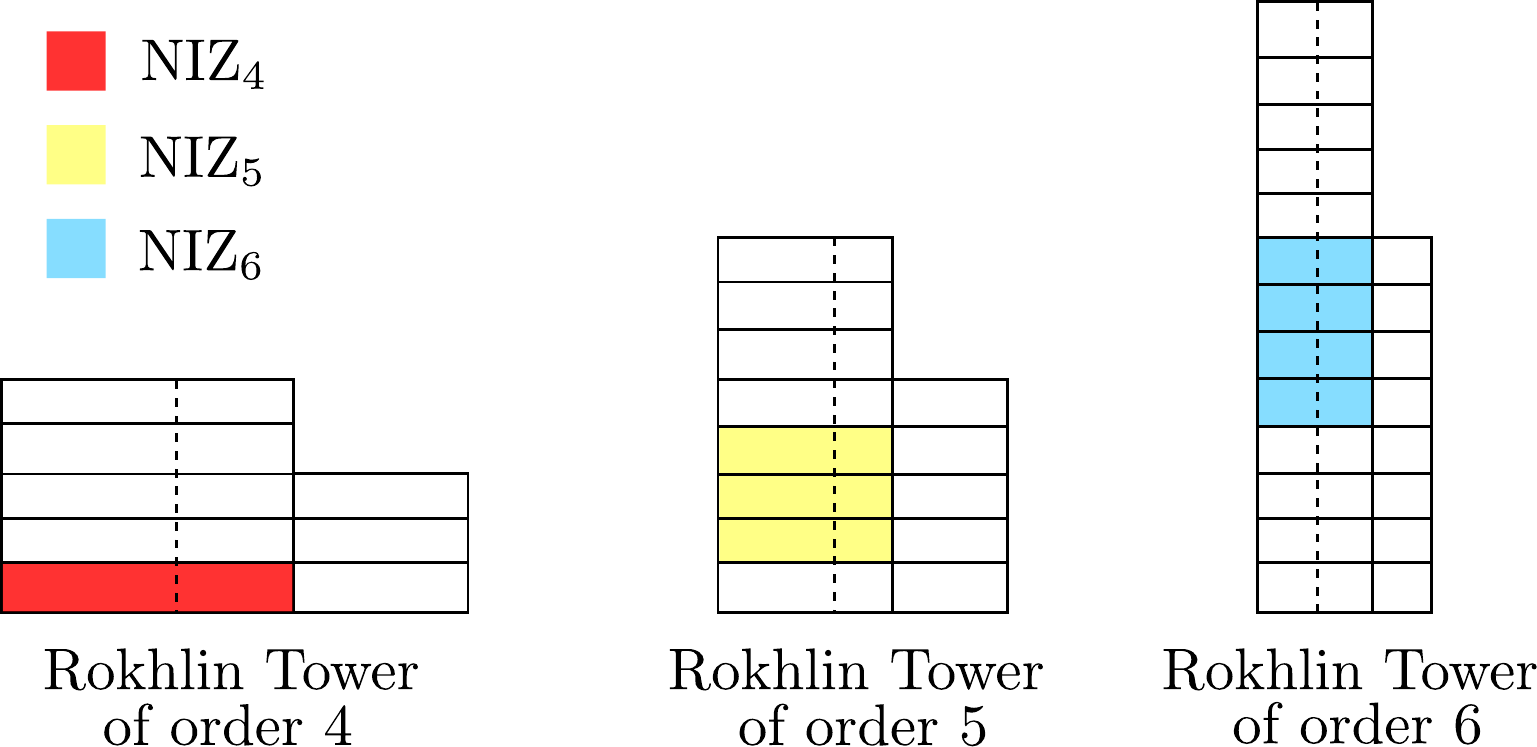}
    \caption{Visualization of $(\NIZ{k})_{k\ge 1}$ when $r=4=\overline{101}$.}
\end{figure}
We also provide a more general figure to understand.
\begin{figure}[H]
    \centering
    \includegraphics[scale=0.35]{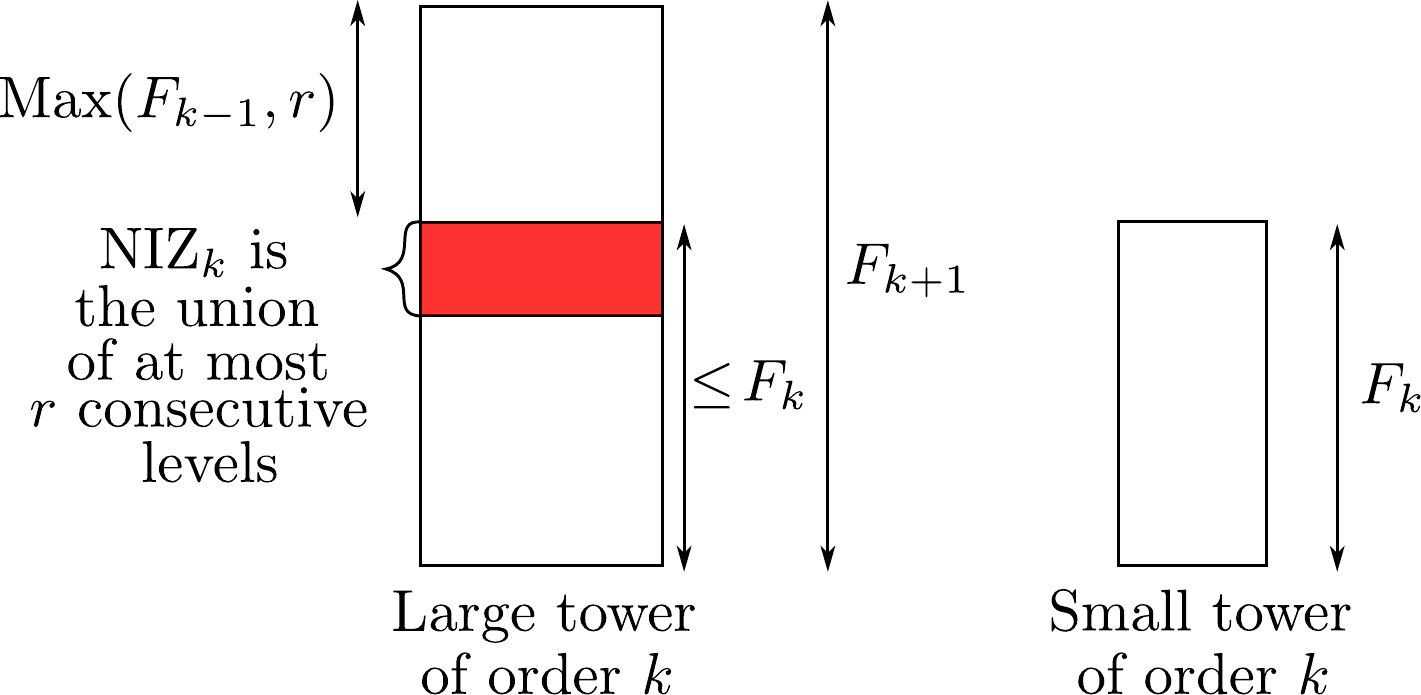}
    \caption{Visualisation of $\NIZ{k}$, $k\ge \ell+2$.}
    \label{NIZ general}
\end{figure}
\end{exemple}
Through the following lemmas, we state some observations on this sequence $(\NIZ{k})_{k\ge 1}$. We start by a description of the cylinders of order $k$ that appear in $\NIZ{k}$ for $k\ge \ell+2$.
\begin{lemme}\label{Contenu NIZ_k}
    If $k\ge 2$, the set $\NIZ{k}$ is composed of cylinders of order $k$ whose name has a pattern $00$ at the leftmost positions.
\end{lemme}
\begin{proof}
    Let $k\ge \ell+2$, consider a cylinder $C_{n_{k+1}\cdots n_2}$ of order $k$ that belongs in $\NIZ{k}$ and where $(n_{k+1},\cdots,n_2)\in\XX_f$. Since this cylinder is part of the large Rokhlin tower of order $k$, $n_{k+1}=0$. Moreover, this cylinder is not in the $F_{k-1}$ highest cylinders of the same tower so $n_k=0$ too.
\end{proof}
Now, we show a relation between cylinders of order $k$ in $\NIZ{k}$ and those of order $k+2$ in $\NIZ{k+2}$ for $k\ge \ell$.
\begin{lemme}\label{relation cylindres NIZ}
    Let $k\ge 2$ and let $(n_{k-1},\cdots,n_2)\in\XX_f$. We have the implication
    \begin{align}
        C_{00n_{k-1}\cdots n_2}\subset \NIZ{k} \implies C_{0010n_{k-1}\cdots n_2}\subset \NIZ{k+2} 
    \end{align}
\end{lemme}
\begin{remarque}
If $k=2$, we agree that the word $n_{k-1}\cdots n_2$ is the empty word.
\end{remarque}
\begin{proof}
    For $r=1$, we let the reader check that $C_{00}\subset \NIZ{2}$ and $C_{0010}\subset \NIZ{4}$. For $r\ge 2$, we have $\NIZ{2}=\emptyset$. So, we can assume $k\ge 3$. Let $C_{00n_{k-1}\cdots n_2}\subset \NIZ{k}$ and $n:=\sum_{i=2}^{k-1} n_iF_i$. Then, $n\in \NIZ{k}\cap [0,F_k[$ where the interval is an integer interval and considered as being part of $\XX$.
    \begin{enumerate}
        \item If $k\ge \ell+2$, we observe that $\NIZ{k}\cap [0,F_k[=[F_k-r,F_k-1[$. It follows that $F_{k+1}+n$, whose Zeckendorf expansion is $\overline{10n_{k-1}\cdots n_2}$, belongs to $[F_{k+2}-r,F_{k+2}-1]\subset \NIZ{k+2}$. Furthermore $F_{k+1}+n\in C_{0010n_{k-1}\cdots n_2}$ which is a level of the large tower of order $k+2$. So  $C_{0010n_{k-1}\cdots n_2}\subset \NIZ{k+2}$.
        \item For $k=\ell$, we observe that $\NIZ{\ell}\cap [0,F_{\ell}[=[0,F_{\ell+1}-r-1[$. It follows that $n+F_{\ell+1}\in[F_{\ell+1},2F_{\ell+1}-r-1]$. We claim that
        \begin{align}\label{etape relation cylindres NIZ} 
            F_{\ell+2}-r \le F_{\ell+1}\le 2F_{\ell+1}-r-1\le  F_{\ell+2}-1.
        \end{align}
        Indeed, since $F_{\ell}\le r < F_{\ell+1}$;
        \begin{itemize}
            \item  $F_{\ell+2}-r\le F_{\ell+2}-F_{\ell}=F_{\ell+1}$,
            \item  $F_{\ell+1}-r-1\ge 0$ and we deduce $2F_{\ell+1}-r-1\ge F_{\ell+1}$,
            \item $2F_{\ell+1}-r-1\le F_{\ell+1}+F_{\ell-1}-1\le F_{\ell+2}-1$.
        \end{itemize}
        From \eqref{etape relation cylindres NIZ}, we get $n+F_{\ell+1}\in [F_{\ell+2}-r,F_{\ell+2}-1]$. Since $[F_{\ell+2}-r,F_{\ell+2}-1]\subset \NIZ{\ell+2}$, we conclude as in the first point.
        \item If $k=\ell+1$, we get $\NIZ{k}\cap [0,F_{\ell}[=[F_{\ell+1}-r,F_{\ell+2}-r-1]$. It follows that $n+F_{\ell+2}\in[F_{\ell+3}-r,2F_{\ell+2}-r-1]$. Since $r\ge F_{\ell}$, $2F_{\ell+2}-r-1\le F_{\ell+3}-1$ and so $n+F_{\ell+2}\in[F_{\ell+3}-r,F_{\ell+3}-1]\subset \NIZ{\ell+2}$. We conclude as in the first point.
    \end{enumerate}
\end{proof}
\begin{lemme}\label{relation par Delta cylindres NIZ}
    Let $k\ge\ell$ and a collection $(n_{k-1},\cdots, n_2)\in\XX_f$ such that $C_{00n_{k-1}\cdots n_2}$ is a cylinder of $\NIZ{k}$. Let $d\in\ZZ$ such that 
    \[\Delta^{(r)}_{\lvert C_{00n_{k-1}\cdots n_2}}=d\]
    then
    \[\Delta^{(r)}_{\lvert C_{0010n_{k-1}\cdots n_2}}=d-1.\]
\end{lemme}

These lemmas imply the following corollary that will give birth to our algorithm.
\begin{corollaire}\label{corollaire pour algo}
    Let $k\ge \ell+4$, the values taken by $\Delta^{(r)}$ on the cylinders of order $k+2$ that compose $\NIZ{k+2}$ are exactly the values taken on cylinders of order $k$ in $\NIZ{k}$ shifted by $-1$.
\end{corollaire}
\begin{proof}[Proof of Lemma~\ref{relation par Delta cylindres NIZ}] 
Let $n:=\sum_{i=2}^{k-1} n_iF_i$ such that $C_{00n_{k-1}\cdots n_2}\subset\NIZ{k}$. We have $C_{00n_{k-1}\cdots n_2}\subset C_{0n_{k-1}\cdots n_2}$ which is, due to the Cut and Stack process, one of the $r$ top levels of the large tower of order $k-1$. Also, due to the Cut and Stack process, $n+r$ must lies in the small tower of order $k-1$ (see Figure~\ref{main argument}) so $(n+r)_k$ must be $1$.
\begin{figure}[H]
    \centering
    \includegraphics[scale=0.30]{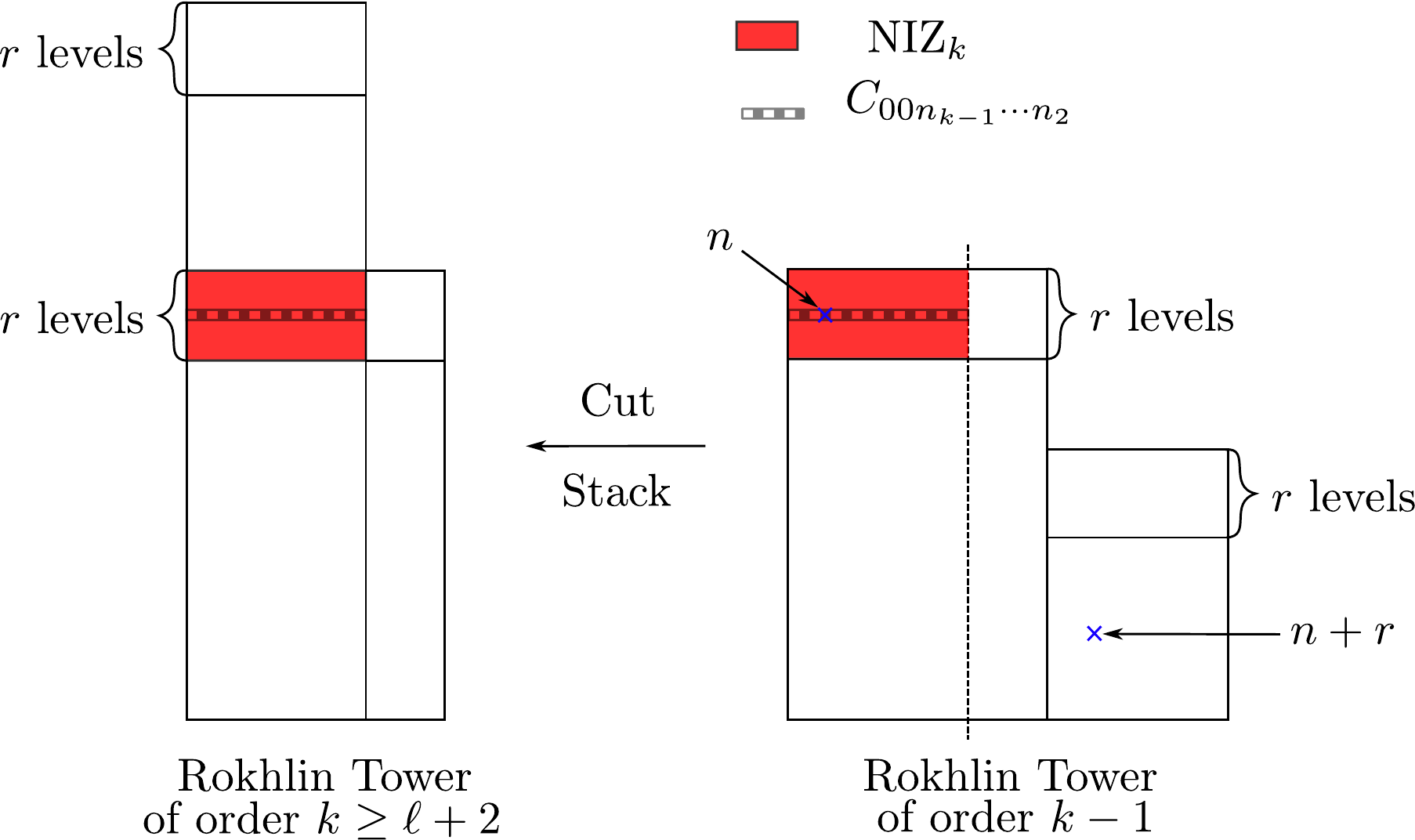}
    \caption{Main argument to justify $(n+r)_k=1$.}
    \label{main argument}
\end{figure}
So, we have the following addition 
    \[\begin{array}{ccccccccc}
        (n)   &     & 0 & 0 & n_{k-1}   &\cdots & n_{\ell}      & \cdots    & n_2 \\
        (r)   & +   &   &   &           &       & r_{\ell}      & \cdots    & r_2\\ \hline
        (n+r) & =   & 0 & 1 & 0         &\cdots & (n+r)_{\ell}  & \cdots    &(n+r)_2
    \end{array}\]
    with, by hypothesis, $\Delta^{(r)}(n)=d$.
    Then $n+F_{k+1}+r$ is calculated as follows
    \[\begin{array}{ccccccccccc}
        (n+F_{k+1})     & 0 & 0 & 1 & 0 & n_{k-1}   &\cdots & n_{\ell}      & \cdots    & n_2 \\
        (r)             & + &   &   &   &   &           &       & r_{\ell}      & \cdots    & r_2\\ \hline
        (n+F_{k+1}+r)   & =   & 0 & 0 & 1 & 1 & 0         &\cdots & (n+r)_{\ell}  & \cdots    &(n+r)_2\\
        (n+F_{k+1}+r)   & =   & 0 & 1 & 0 & 0 & 0         &\cdots & (n+r)_{\ell}  & \cdots    &(n+r)_2
    \end{array}\]
    The sum of digits of $n+F_{k+1}+r$ is the same as those of $n+r$ shifted by $-1$ due to the correction of the expansion made. Hence, we get $\Delta^{(r)}(n+F_{k+1})=\Delta^{(r)}(n)-1$ i.e.     
    \[\Delta^{(r)}_{\lvert C_{0010n_{k-1}\cdots n_2}}=\Delta^{(r)}_{\lvert C_{00n_{k-1}\cdots n_2}}-1.\]
\end{proof}
\begin{proof}[Proof of Corollary~\ref{corollaire pour algo}]
Definition~\ref{Def NIZ_k} ensures that the same number of cylinders of order $k+2$ in $\NIZ{k+2}$ is equal to the number of cylinders of order $k$ in $\NIZ{k}$ : there are $r$ cylinders of each order. Lemma~\ref{relation cylindres NIZ} and Lemma~\ref{relation par Delta cylindres NIZ} enable us to conclude.
\end{proof}
These results enable us to give an algorithm that will compute the value of $\mu^{(r)}(d)$ for any $d\in\ZZ$. This algorithm has a low complexity (polynomial). Indeed, Corollary~\ref{corollaire pour algo} implies that we can anticipate the values that will appear in the New Information Zone of order $k$. So we just need to compute $\Delta^{(r)}$ on a finite number of integers to compute exactly $\mu^{(r)}(d)$ ($d\in\ZZ$). We will precise this fact later in this Section. Before, we give a pseudo-code and provide an example of computation.
    
    \subsection{Pseudo-code}\label{PC}

\begin{algorithm}[H]
\caption{Compute $\mu^{(r)}(d)$}
\begin{algorithmic} 
\REQUIRE $r \geq 1$ and  $d\in\ZZ$
\ENSURE $\mu^{(r)}(d)$
\STATE \textbf{STEP 0}: Find the unique $\ell\ge 1$ such that $F_{\ell}\le r < F_{\ell+1}$
\STATE \textbf{STEP 1}: Passage in $\NIZ{\ell}$
\STATE Define $App_{\ell}\leftarrow 0$ (number of apparition of $d$ during the for-loop)
\FOR{$n=0,\cdots,F_{\ell+1}-r-1$}
    \IF{$\Delta^{(r)}(n)=d$}
        \STATE $App_{\ell}\leftarrow App_{\ell}+1$
    \ENDIF
\ENDFOR
\STATE \textbf{STEP 2}: Passage in $\NIZ{\ell+1}$
\STATE Define $App_{\ell+1}\leftarrow 0$
\FOR{$n=F_{\ell+1}-r\cdots,F_{\ell+2}-r-1$}
    \IF{$\Delta^{(r)}(n)=d$}
        \STATE $App_{\ell+1}\leftarrow App_{\ell+1}+1$
    \ENDIF
\ENDFOR
\STATE \textbf{STEP 3}: Passage in $\NIZ{\ell+2}$
\STATE Define $App_{\ell+2}\leftarrow 0$
\FOR{$n=F_{\ell+2}-r,\cdots, F_{\ell+2}-1$}
    \STATE Store $\Delta^{(r)}(n)$ in an array \textbf{ARR1} (with multiplicity)
    \IF{$\Delta^{(r)}(n)=d$}
        \STATE $App_{\ell+2}\leftarrow App_{\ell+2}+1$
    \ENDIF
\ENDFOR
\STATE \textbf{STEP 4}: Passage in $\NIZ{\ell+3}$
\STATE Define $App_{\ell+3}\leftarrow 0$
\FOR{$n=F_{\ell+3}-r,\cdots, F_{\ell+3}-1$}
    \STATE Store $\Delta^{(r)}(n)$ in an array \textbf{ARR2} (with multiplicity)
    \IF{$\Delta^{(r)}(n)=d$}
        \STATE $App_{\ell+3}\leftarrow App_{\ell+3}+1$
    \ENDIF
\ENDFOR
\end{algorithmic}
\end{algorithm}

\newpage

\begin{algorithm}[H]
\begin{algorithmic}
\STATE \textbf{STEP 5}: Search for $d$ in higher orders
\FOR{$i\in \textbf{ARR1}$}
    \IF{$i>d$}
        \STATE $d$ will appear at the order $\ell+2+2(i-d)$ so
        \IF{$App_{\ell+2+2(i-d)}$ is defined}
            \STATE $App_{\ell+2+2(i-d)}\leftarrow App_{\ell+2+2(i-d)}+1$
            \ELSE 
            \STATE $App_{\ell+2+2(i-d)}\leftarrow 1$
        \ENDIF     
    \ENDIF
\ENDFOR
\FOR{$i\in \textbf{ARR2}$}
    \IF{$i>d$}
        \STATE $d$ will appear at the order $\ell+3+2(i-d)$ so
        \IF{$App_{\ell+3+2(i-d)}$ is defined}
            \STATE $App_{\ell+3+2(i-d)}\leftarrow App_{\ell+3+2(i-d)}+1$
            \ELSE 
            \STATE $App_{\ell+3+2(i-d)}\leftarrow 1$
        \ENDIF     
    \ENDIF
\ENDFOR
\STATE \textbf{STEP 6}: Compute $\mu^{(r)}(d)$
\STATE $\mu^{(r)}(d)\leftarrow \sum_{k\ge \ell} \frac{App_k}{\varphi^k}$
\end{algorithmic}
\end{algorithm}

    \subsection{An example : computation of \texorpdfstring{$\mu^{(4)}$}{u4}}

Let us compute the measure $\mu^{(r)}$ for $r=4$. We have $\ell=4$. Looking at the Rokhlin towers (see Example~\ref{Example NIZ pour r=4}) gives that 
\begin{itemize}
    \item $\NIZ{k}=\emptyset$ if $k=1,2,3$,
    \item $\NIZ{4}=C_{0000}$,
    \item $\NIZ{5}=C_{00001}\sqcup C_{00010} \sqcup C_{00100}$,
    \item $\NIZ{6}=C_{000101} \sqcup C_{001000} \sqcup C_{001001} \sqcup C_{001010}$,
    \item $\NIZ{7}=C_{0010000} \sqcup C_{0010001} \sqcup C_{0010010} \sqcup C_{0010100}$.
\end{itemize}
The next terms of $\NIZ{k}$ are found using Lemma~\ref{relation cylindres NIZ}. Using Corollary~\ref{corollaire pour algo}, one can construct the following table that gives the values taken by $\Delta^{(r)}$ (with multiplicity) on each cylinder of $\NIZ{k}$.

\begin{figure}[H]
    \centering
    \includegraphics[scale=0.35]{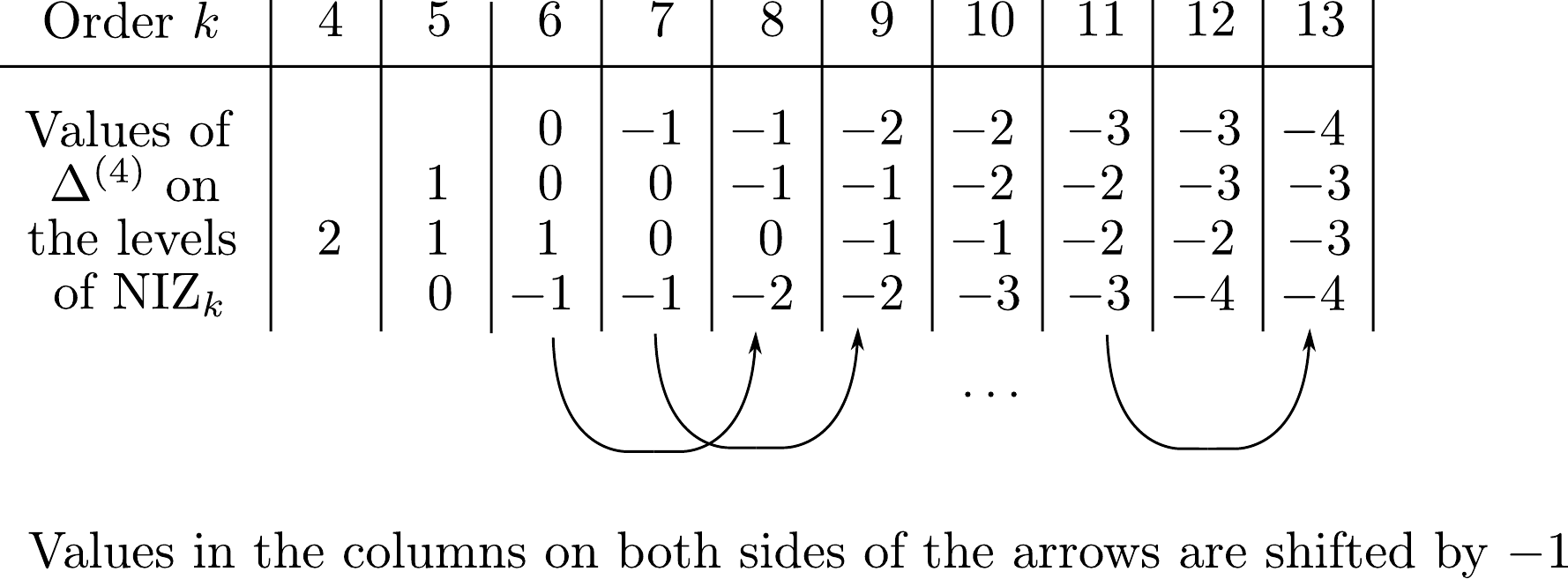}
    \caption{Table of the values of $\Delta^{(4)}$ on $\NIZ{k}$.}
    \label{Value of Delta r}
\end{figure}
Thus, it is possible to know how many times some integer $d\in\ZZ$ appear in each column (ie at each order) and, since we know the measure of a cylinder at each order, it is possible to deduce the value of $\mu^{(r)}(d)$. For instance
\begin{itemize}
    \item The value $2$ appears once at order $4$ so $\mu^{(4)}(2)=\frac{1}{\varphi^4}$.
    \item The value $1$ appears twice at order $5$, once at order $6$ so $\mu^{(4)}(1)=\frac{2}{\varphi^5} + \frac{1}{\varphi^6}$.
    \item The value $0$ appears once at order $5$ and $8$. It also appears twice at order $6$ and $7$, we deduce that $\mu^{(4)}(0)=\frac{1}{\varphi^5} + \frac{2}{\varphi^6} + \frac{2}{\varphi^7} + \frac{1}{\varphi^8}$.
    \item The value $-1$ appears once at order $6$ and $10$. It appears twice at order $7$, $8$ and $9$. We deduce $\mu^{(4)}(-1)=\frac{1}{\varphi^6} + \frac{2}{\varphi^7} + \frac{2}{\varphi^8} + \frac{2}{\varphi^9} + \frac{1}{\varphi^{10}}$.
    \item Since $-1$ is a value at order $\ell+2=6$ that does not appear in the previous order, we deduce by Corollary~\ref{corollaire pour algo} that the value $-2$ will appear as many times as $-1$ appear but at orders incremented by $2$. In other words, the value $-2$ appears once at order $8$ and $12$. It appears twice at order $9$, $10$ and $11$. Thus $\mu^{(4)}(-2)=\mu^{(4)}(-1)\times \frac{1}{\varphi^2}$. We observe that $-2$ does not appear in the lower orders. This observation will be better explained in Proposition~\ref{conseq algo 1} and Corollary~\ref{corollaire pour algo2}.
\end{itemize}
So, we get after simplification 
\begin{align*}
    \mu^{(4)}(d)=\left\{\begin{array}{ll}
         0& \textrm{ if } d>2  \\
         \frac{1}{\varphi^4} & \textrm{ if } d=2\\
         \frac{1}{\varphi^3} & \textrm{ if } d=1\\
         \frac{2}{\varphi^4} & \textrm{ if } d=0\\
         \mu^{(4)}(-1)\times \varphi^{2d+2}& \textrm{ if } d<0\\
    \end{array}\right.
\end{align*}
(where $\mu^{(4)}(-1)=\frac{1}{\varphi^4}+\frac{1}{\varphi^{6}}$.)
\begin{figure}[H]
    \centering
    \includegraphics[scale=0.24]{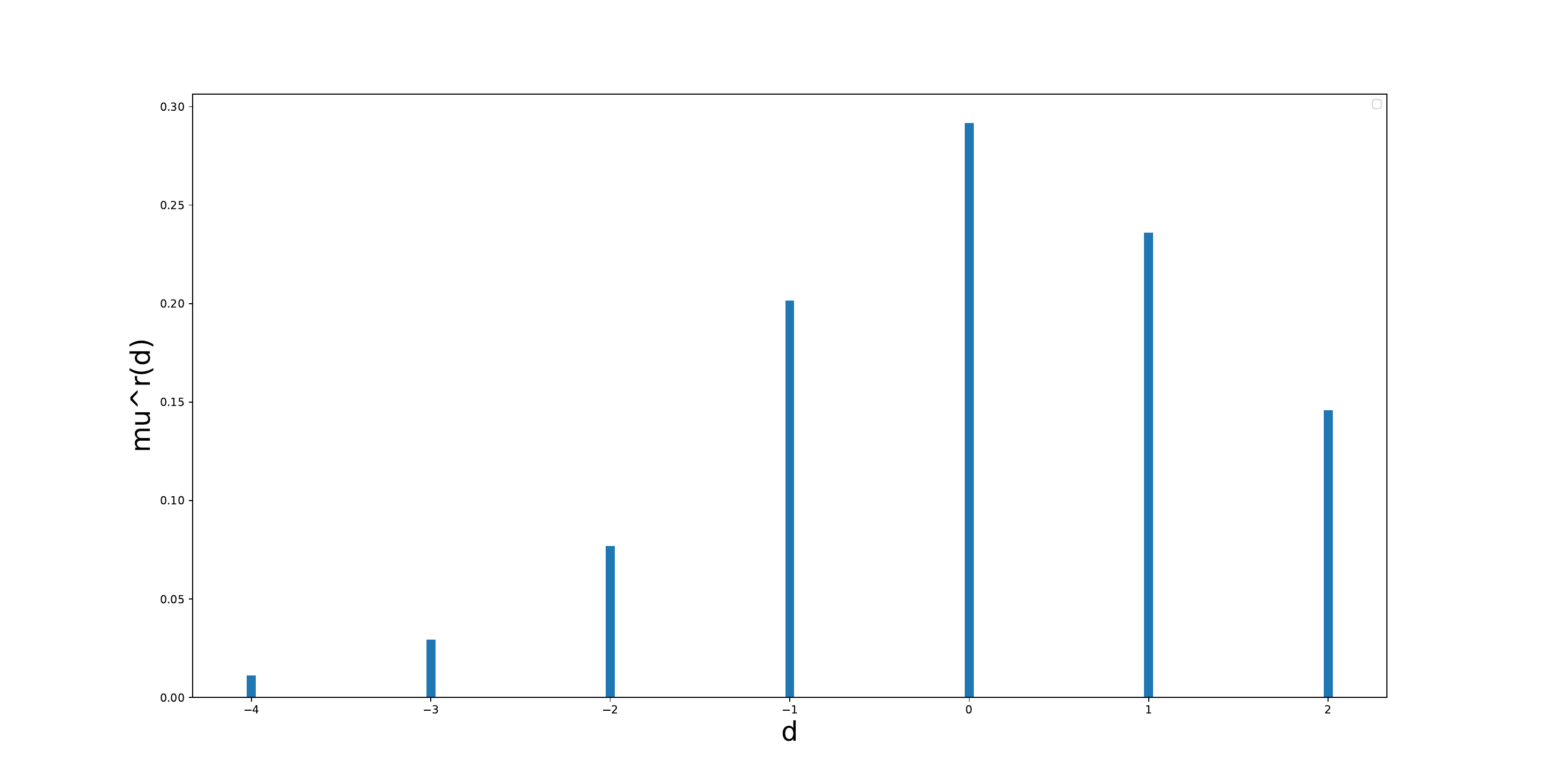}
    \caption{Bar chart of $\mu^{(4)}$.}
\end{figure}
As an exercise, the reader can check that $\mu^{(1)}=\mu^{(F_3)}=\mu^{(F_4)}$. In the next Section, we are going to prove that, for any $k\ge 2$ 
\[\mu^{(F_k)}=\mu^{(1)}.\]
    
    \subsection{Remarks on the algorithm}\label{ROTA}
    
Once again, the algorithm described in Subsections~\ref{DOTA} and \ref{PC} can be adapted in integer base so the following observations can also be adapted in integer base. Also, we write an algorithm that returns, for a given $r\ge 1$ and $d\in\ZZ$, the value $\mu^{(r)}(d)$. Of course, the algorithm can also be adapted to return $\mu^{(r)}(d)$ for every $d\in\ZZ$ or, at least to be executed on a computer, for a finite number of integers $d$ without more computations (the adaptation only consists in manipulations of arrays). 

In this subsection, let $r\ge 1$. Let also $\ell\ge 1$ be the unique integer such that $F_{\ell}\le r< F_{\ell+1}$. The first important observation due to the algorithm is that we only need to compute $\Delta^{(r)}(n)$ for a finite number of $n$'s to know exactly the quantity $\mu^{(r)}$ . More precisely, we need to compute $\Delta^{(r)}$ for 
\begin{itemize}
    \item $F_{\ell+1}-r$ integers during \textbf{STEP 1},
    \item $F_{\ell}$ integers during \textbf{STEP 2},
    \item $r$ integers during \textbf{STEP 3} and
    \item $r$ integers during \textbf{STEP 4}.
\end{itemize}
So, we need to compute the image of $F_{\ell+2}+r$ different integers to compute $\mu^{(r)}$ exactly. One can prove that, when $r=F_{\ell}$, it is possible to adapt the algorithm to reduce that number to $F_{\ell+2}$. The adaptation consists into forgetting \textbf{STEP 4} and store the $F_{\ell}$ values in \textbf{STEP 2}.

The algorithm implies the following properties.
\begin{prop}\label{conseq algo 1}
If $d$ is small enough in $\ZZ$ then there exist exactly $2r$ cylinders of (non-necessarily distinct) orders $k_1,\cdots,k_{2r}$, respectively in $\NIZ{k_1},\cdots,\NIZ{k_{2r}}$ on which $\Delta^{(r)}$ takes the value $d$.
\end{prop}
\begin{remarque}
    The ``small enough''  assumption is precised in the proof.
\end{remarque}
It leads to the corollary~\ref{corollaire pour algo2} we state again here.
\begin{corollaire*}
    For $d$ small enough in $\ZZ$, we have the formula
    \[\mu^{(r)}(d-1)=\mu^{(r)}(d)\times \frac{1}{\varphi^{2}}.\]
\end{corollaire*}
\begin{proof}[Proof of Proposition~\ref{conseq algo 1}]
    For $k\ge \ell +2$, there are $r$ cylinders of order $k$ that compose $\NIZ{k}$ so $\Delta^{(r)}$ takes, at most, $r$ different values on those cylinders. For $k\ge \ell+2$, let $(d_i^{(k)})_{1\le i\le r}$ be the collection of values taken by $\Delta^{(r)}$ on the cylinders of order $k$ in $\NIZ{k}$, repeated with multiplicity. Denote \[m:=min\left\{d_i^{(k)}~|~i=1,\cdots,r~\textrm{and}~k=\ell+2,\ell+3\right\}.\]
    Now, let $d\le m$. We observe that due to Lemma~\ref{Contenu NIZ_k} and Lemma~\ref{relation par Delta cylindres NIZ}, $d$ is actually smaller than any value taken by $\Delta^{(r)}$ on $\NIZ{\ell}$ or $\NIZ{\ell+1}$. Furthermore, for any $i\in [1,r]$ and any $k\in\{\ell+2,\ell+3\}$, there exists a unique $j_i^{(k)}\in\NN$ such that $d=d_i^{(k)}-j_i^{(k)}$. Due to Corollary~\ref{corollaire pour algo}, $d$ will be a value reached by $\Delta^{(r)}$ on the corresponding cylinder of order $k+2j_i^{(k)}$ (by repeated use of Lemma~\ref{etape relation cylindres NIZ} and \ref{relation par Delta cylindres NIZ}). There are $2r$ cylinders at order $\ell+2$ and $\ell+3$ so $d$ will appear exactly on $2r$ cylinders.
\end{proof}
\begin{proof}[Proof of Corollary~\ref{corollaire pour algo2}]
    With the same notation as in the proof of Proposition~\ref{conseq algo 1}, let $d\le m$. Then $d$ is a value taken by exactly $2r$ different cylinders of some orders. But, due to Proposition~\ref{conseq algo 1} and Corollary~\ref{corollaire pour algo}, $\Delta^{(r)}$ will take the value $d-1$ on the same number of cylinders but with the orders of those cylinders shifted by $+2$ so their measures are divided by $\varphi^{2}$. 
\end{proof}

\section{How to prove \texorpdfstring{$\mu^{(F_{\ell})}=\mu^{(1)}$}{uFk=u1}}\label{HTPMFM1}

There is an analogous relation in the integer base $b$ case. The law of adding $b^{\ell}$ is the same as the one of adding $1$. It is trivial in base $b$ since the addition $x+b^{\ell}$ will not change the first digits of $x$, the addition really consists in adding $1$ from a certain position. In the Zeckendorf representation case, it is not trivial since carries propagate in both directions. Here, the proof consists into looking at our algorithm described in Section~\ref{HTCM} in the special case $r=F_{\ell}$. However, before starting, we need to compute $\mu^{(1)}$.
\begin{prop}\label{mu1}
For $d\in\ZZ$, we have
\[\mu^{(1)}(d)=\left\{\begin{array}{ll}
    0 & \textrm{ if }d\ge 2  \\
    \frac{1}{\varphi^2} & \textrm{ if } d=1 \\
    \frac{1}{\varphi^{2-2d}} & \textrm{ otherwise.}
\end{array}\right.\]
\end{prop}
\begin{proof}
    Thanks to the algorithm, we explicit the following partition:
    \begin{align}\label{partition pour Delta 1}
        \XX\backslash \{(01)^{\infty},(10)^{\infty}\}=C_{00}\bigsqcup_{d\le 0} \left(C_{0(01)^{1-d}}\bigsqcup C_{00(10)^{1-d}}\right).
    \end{align}
    We observe that $\Delta^{(1)}(C_{00})=1$ and, for every $d\le 0$, 
    \[\Delta^{(1)}(C_{0(01)^{1-d}})=\Delta^{(1)}(C_{00(10)^{1-d}})=d.\] 
    Using Proposition~\ref{prop caract mesure zeck}, we conclude the proof.
\end{proof}

Now, let $\ell\ge 3$. We follow the path given by the algorithm. We can rewrite Proposition~\ref{Def NIZ_k} in our special case $r=F_{\ell}$ :
\begin{enumerate}
    \item if $2\le k\le \ell-1$, $\NIZ{k}=\emptyset$,
    \item $\NIZ{\ell}$ is the union of the $F_{\ell-1}$ first levels of the large tower of order $\ell$ and
    \item if $k\ge \ell+1$, $\NIZ{k}$ is the union of the $F_{\ell}$ levels between the $F_k-F_{\ell}^{th}$ and the $F_k^{th}$ levels of the large tower of order $k$.
\end{enumerate}
Executing the algorithm, we want to compute $\Delta^{(F_{\ell})}$ on the levels of $\NIZ{\ell}$. We obtain the following lemma.
\begin{lemme}\label{DeltaFl_orderl}
    \[\Delta^{(F_{\ell})}(\NIZ{\ell})=1.\]
\end{lemme}
\begin{proof}
Indeed, consider a cylinder of order $\ell$ in $\NIZ{\ell}$. Its name has for leftmost digits $000$ since the cylinders contains one element of the integer interval $[0,F_{\ell-1}-1]$. So the addition with $F_{\ell}$ considered in these cylinders are
\[\begin{array}{ccccccc}
     (x)            &   &\cdots & 0 & 0 & 0 & \cdots  \\
     (F_{\ell})     & + &       &   & 1 &   & \\ \hline
     (x+F_{\ell})   & = &\cdots & 0 & 1 & 0 & \cdots
\end{array}\]
Thus the variation of the sum of digits is $1$.
\end{proof}
We now look at $\NIZ{\ell+1}$ which is composed of $F_{\ell}$ cylinders of order $\ell+1$. We have the following proposition.
\begin{prop}\label{DeltaFl_orderl+1}
In $\NIZ{\ell+1}$, there are 
\begin{itemize}
    \item $F_{\ell-1}$ levels on which $\Delta^{(F_{\ell})}=0$ and
    \item $F_{\ell-2}$ levels on which $\Delta^{(F_{\ell})}=1$.
\end{itemize}
\end{prop}
For technical reasons, we need to separate the proof in two parts : one when $\ell$ is odd and the other if $\ell$ is even. However, since both are treated the same way (the differences are only technicalities), we only consider the case where $\ell$ is odd. We are going to prove that the values taken by $\Delta^{(F_{\ell})}$ are the $F_{\ell}$ first terms of the sequence 
\begin{itemize}
    \item $1,0,1,0,0,1,1,1,0,0,0,0,0,1,1,1,1,1,1,1,1, \cdots$ (if $\ell$ is odd)
    \item  $0,1,0,1,1,0,0,0,1,1,1,1,1,0,0,0,0,0,0,0,0,\cdots$ (if $\ell$ is even) or
\end{itemize}
where both sequences start with $0$ or $1$ depending on the parity and then, the construction consists in concatenating alternatively $0$'s or $1$'s a Fibonacci number of times.
\begin{proof}[Proof of Proposition~\ref{DeltaFl_orderl+1} if $\ell$ is odd.]
    Let $\ell'\ge 1$ and $\ell:=2\ell'+1\ge 3$ . The highest level in $\NIZ{\ell+1}$ is the cylinder $C_{00(10)^{\ell'}}$ since it contains $F_{\ell+1}-1$ ; the lowest is the cylinder $C_{00010^{\ell-3}}$ since it contains $F_{\ell-1}$.

    The cylinder of order $\ell+1$ just below $C_{00(10)^{\ell'}}$ is $C_{00(10)^{\ell'-1}01}$. We observe that, if $\ell=3$ (ie $\ell'=1$ or $F_{\ell}=2$), this last cylinder $C_{00(10)^{\ell'-1}01}$ is actually the same as $C_{00010^{\ell-3}}$. In general, there are two possibilities.
    \begin{itemize}
        \item Either there are no other cylinders of order $\ell+1$ which means $\ell=3$ (ie $\ell'=1$ or $F_{\ell}=2$). So far, we have mentioned $2$ cylinders of order $4$ in $\NIZ{4}$ and the values of $\Delta^{(F_3)}$ on these cylinders are $0$ and $1$ (with an easy computation): it is the conclusion of the proposition if $\ell=3$ (if we agree $F_0:=1$).
        \item Or there are other cylinders to find in $\NIZ{\ell+1}$ which means  $F_{\ell}>2$ (ie $\ell\ge 5$ or $F_{\ell}\ge 5$) and implies that there are $F_{\ell}-F_2\ge 3$ levels to identify.
    \end{itemize}
    To continue, we suppose $\ell\ge 5$ (ie $\ell'\ge2$) and identify the three cylinders that are below $C_{00(10)^{\ell'-1}01}$ : they are
    \begin{itemize}
        \item $C_{00(10)^{\ell'-1}00}$ on which $\Delta^{(F_{\ell})}$ equals 1,
        \item $C_{00(10)^{\ell'-2}0101}$ on which $\Delta^{(F_{\ell})}$ equals 0 and
        \item $C_{00(10)^{\ell'-2}0100}$ on which $\Delta^{(F_{\ell})}$ equals 0.
    \end{itemize}
    Once again, we observe that, if $\ell=5$ (ie $\ell'=2$), $C_{00(10)^{\ell'-2}0100}=C_{00010^{\ell-3}}$. There are again two possibilities.
    \begin{itemize}
        \item Either there is no other cylinders of order $\ell+1$ which means $\ell=5$ (ie $\ell'=2$ or $F_{\ell}=5$). So far again, we have mentioned $5$ cylinders of order $6$ in $\NIZ{6}$ and $\Delta^{(F_5)}$ takes thrice the value $0$ and twice the value $1$ on these cylinders : it is the conclusion of the proposition if $\ell=5$.
        \item Or there are other cylinders to find in $\NIZ{\ell+1}$ which means  $F_{\ell}>5$ (ie $\ell\ge 7$ or $F_{\ell}\ge 13$) and implies that there are $F_{\ell}-F_5\ge 8$ levels to identify.
    \end{itemize}
    We continue the same way. If we assume there are, for some $k\in[1,\ell-1]$, $F_{\ell}-F_{2k-1}\ge F_{2k}$ cylinders. There are two kind of that cylinders.
    \begin{itemize}
        \item Those whose names are $C_{00(10)^{\ell'-k+1}00n_{2k-3}\cdots n_2}$ for all $(n_{2k-3}\cdots n_2)\in\XX_f$ : there are $F_{2k-2}$ such cylinders and one can compute that $\Delta^{(F_{\ell})}$ is $1$.
        \item Those whose names are $C_{00(10)^{\ell'-k}010n_{2k-2}\cdots n_2}$ for all $(n_{2k-2}\cdots n_2)\in\XX_f$ : there are $F_{2k-1}$ such cylinders and one can compute that $\Delta^{(F_{\ell})}$ is $0$.
    \end{itemize}
    We observe that if $\ell=2k+1$ (ie $\ell'=k$) the last cylinder identified is $C_{00(10)^{\ell'-k}010^{2k-2}}$ and is actually the same as $C_{00010^{\ell-3}}$. So far, we have identified $1+F_1+(F_2+F_3)+(F_4+F_5)+\cdots +(F_{2k-2}+F_{2k-1})=F_{2k+1}$ cylinders and 
    \begin{itemize}
        \item on $1+F_2+F_4+F_{2k-2}= F_{2k-1}$ of them, $\Delta^{(F_{\ell})}$ is $1$
        \item on $F_1+F_3+F_5+\cdots +F_{2k-1}=F_{2k}$ of them, $\Delta^{(F_{\ell})}$ is $0$
    \end{itemize}
    It is exactly the conclusion when $\ell=2k+1$.
\end{proof}

Finally, we look at $\NIZ{\ell+2}$ which contains also $F_{\ell}$ cylinders of order $\ell+2$. 
\begin{prop}\label{DeltaF1_orderl+2}
 \[\Delta^{(F_{\ell})}(\NIZ{\ell+2})=0.\]
\end{prop}
\begin{proof}
    Consider a cylinder of order $\ell+2$ in $\NIZ{\ell+2}$. Since it contains exactly one integer of the interval $[F_{\ell+1},F_{\ell+2}-1]$, its name is $C_{0010n_{\ell-1}\cdots n_2}$ for every $(n_{\ell-1}\cdots n_2)\in\XX_f$. We compute that, on such a cylinder, $\Delta^{(F_{\ell})}$ is null.
\end{proof}
We do not need to prove anything for $\NIZ{\ell+3}$, the algorithm teaches us that $\Delta^{(F_{\ell})}$ is going to take the same values, but shifted by $-1$, on the same number of levels as in $\NIZ{\ell+1}$. Now it is time to prove our theorem.
\begin{theo*}
For any $\ell\ge 2$, $\mu^{(F_{\ell})}=\mu^{(1)}$.
\end{theo*}
\begin{proof}
    We recall 
    \[\mu^{(1)}(d)=\left\{\begin{array}{ll}
    0 & \textrm{ if }d\ge 2  \\
    \frac{1}{\varphi^2} & \textrm{ if } d=1 \\
    \frac{1}{\varphi^{2-2d}} & \textrm{ otherwise.}
\end{array}\right.\]
    For $\ell=2$, it is trivial. Let $\ell\ge 3$ and $d\in\ZZ$. It is also trivial when $d\ge 2$ since $\Delta^{(F_{\ell})}$ does not take this value. Suppose that $d=1$, this value appears in $F_{\ell-1}$ levels in $\NIZ{\ell}$, $F_{\ell-2}$ levels in $\NIZ{\ell+1}$ and nowhere else. So
    \[\mu^{(F_{\ell})}(1)=\frac{F_{\ell-1}}{\varphi^{\ell}}+\frac{F_{\ell-2}}{\varphi^{\ell+1}}=\frac{\varphi^{\ell-1}}{\varphi^{\ell+1}}=\frac{1}{\varphi^{2}}=\mu^{(1)}(1).\]
    If $d\le 0$, the value $d$ will appear, due to Corollary~\ref{corollaire pour algo}, on $F_{\ell-1}$ levels at order $\ell+1-2d$, $F_{\ell}$ levels at order $\ell+2-2d$, $F_{\ell-2}$ levels at order $\ell+3-2d$ and nowhere else. Thus
    \begin{align*}
        \mu^{(F_{\ell})}(1)&=\frac{F_{\ell-1}}{\varphi^{\ell+1-2d}}+\frac{F_{\ell}}{\varphi^{\ell+2-2d}}+\frac{F_{\ell-2}}{\varphi^{\ell+3-2d}}\\
            &=\frac{\varphi^{\ell+1}}{\varphi^{\ell+3-2d}}\\
            &=\frac{1}{\varphi^{2-2d}}\\
            &=\mu^{(1)}(d).
    \end{align*}
\end{proof}

\section{\texorpdfstring{$\Delta^{(r)}$}{Dr} as a mixing process}\label{DMP Zeck}

We work on the probability space $(\XX,\mathscr{B}(\XX),\PP)$. For a given integer $r$, $\Delta^{(r)}$ is viewed as a random variable with law $\mu^{(r)}$ by Corollary~\ref{corollaire_prop_moment Zeck} (the randomness comes from the argument $x$ of $\Delta^{(r)}$, considered as a random outcome in $\XX$ with law $\PP$). In this section, we will decompose $\Delta^{(r)}$ as a sum composed of a finite number of random variables satisfying a universal inequality on their mixing coefficients, as in $\cite{TREJYH}$.

    \subsection{The process}\label{the process}

The case $r=0$ is irrelevant so we assume $r\ge 1$. For $1\le i\le \rho(r)$, we will write $B_i$ as the $i^{th}$ block present in the expansion of $r$, starting from the unit digit. We define $r[i]$ as the integer whose expansion is given by the $i$ first blocks of the expansion of $r$ (see Figure below). We observe that $r[\rho(r)]=r$.
\begin{figure}[H]
    \centering
    \includegraphics[scale=0.38]{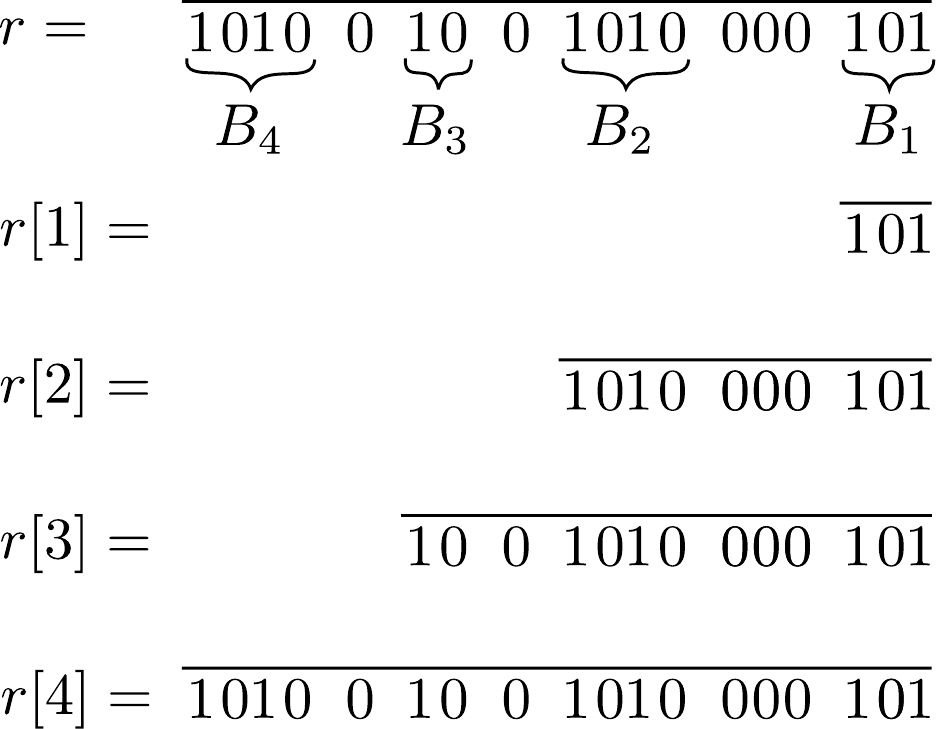}
    \caption{Example of the construction of $(r[i])_{i=0,\cdots,\rho(r)}$.}
    \label{construction r[i] Zeck}
\end{figure} 
With the convention $r[0]:=0$, we observe the trivial equality 
\[r=\sum_{i=1}^{\rho(r)}r[i]-r[i-1].\]
For $1\le i\le \rho(r)$, we define almost everywhere on $\XX$ (see Subsection~\ref{SOzeck})
\[X_i^{(r)}:=\Delta^{(r[i]-r[i-1])}\circ T^{r[i-1]}.\] 
Since $r[i]-r[i-1]=\overline{B_i 0\cdots 0}$, the function  $X_i^{(r)}$ is a random variable corresponding to the action of the $i^{th}$ block $B_i$ once the previous blocks have already been taken into consideration.
From \eqref{decomp_sum Zeck}, we get
\[\Delta^{(r)}=\sum_{i=1}^{\rho(r)}X_i^{(r)}.\]
In particular, if $x\in\XX$ is randomly chosen with law $\PP$, then $\displaystyle\sum_{i=1}^{\rho(r)}X_i^{(r)}(x)$ follows the law $\mu^{(r)}$.

    \subsection{The \texorpdfstring{$\alpha$}{a}-mixing coefficients on the actions of blocks}

This part is devoted to the proof of one of the main theorem stated in the Introduction: we are going to show that the $\alpha$-mixing coefficients for the process $(X_i^{(r)})_{i=1,\cdots,\rho(r)}$ satisfy a universal upper bound which is independent of $r$. In particular, these coefficients exponentially decrease to $0$.
\begin{theo*}
 The $\alpha$-mixing coefficients of $(X_i^{(r)})_{i=1,\cdots,\rho(r)}$ satisfy
\[\forall k\ge 1, \hspace{5mm} \alpha(k)\le 12\left(1-\frac{1}{\varphi^8}\right)^{\frac{k}{6}}+\frac{1}{\varphi^{2k}}.\]
\end{theo*}

\begin{proof}
    Let $k,p\ge 1$. Let $A\in\sigma(X_i^{(r)}:1\le  i\le p)$ and $B\in\sigma(X_i^{(r)}:  i\ge k+p)$. Let $C\in\sigma(X_i^{(r)}:  p< i <k+p)$ such that $\PP(C)>0$. Then, we have the following inequality
     \begin{align}\label{ineg triang 2 Zeck}
        & \begin{array}{lll}
        \lvert\PP(A\cap B)-\PP(A)\PP(B)\rvert & \le \lvert\PP(A\cap B)-\PP_C(A\cap B)\rvert \\
        & \hspace{1 cm}+ \lvert\PP_C(A\cap B)-\PP_C(A)\PP_C(B)\rvert \\
        & \hspace{1 cm} + \lvert\PP_C(A)\PP_C(B)-\PP(A)\PP(B)\rvert.
        \end{array}
    \end{align}
    Using \eqref{ineg generale conditionnement Zeck}, we get
    \begin{align}\label{etape 1 mixing block}
        \lvert\PP(A\cap B)-\PP(A)\PP(B)\rvert & \le 3\PP(\overline{C}) + \lvert\PP_C(A\cap B)-\PP_C(A)\PP_C(B)\rvert.
    \end{align}
    As in \cite{TREJYH} (Lemma 4.3), we will consider an event $C$ which has a probability close to $1$ and such that, conditionally to $C$, the events $A$ and $B$ are ``almost'' independent. A difference with \cite{TREJYH} is that, as the digits of a random Zeckendorf-adic numbers are not independent, the last term at the right-hand side of \eqref{etape 1 mixing block} will not totally vanish and, thus, some extra work is needed. However, we can anticipate that this problem of ``almost'' independence will be solved using the inequality on the $\phi$-mixing coefficients (on the law of coordinates) stated in Proposition~\ref{Melange_coord}.
    
    To define the event $C$, we use the ideas developed in Subsection~\ref{subsection Focus} (especially Corollaries~\ref{stopping condition bloc 1} and ~\ref{stopping cdt bloc >1}). Since we are working with blocks, we introduce $\ell_i$ as the number of patterns $10$ in $B_i$, the $i^{th}$ block of $r$. We also introduce $n_i$ as the minimal index of digits in $B_i$. Then, we define, for $i=1,\cdots, \rho(r)$, the set Adm$(i)$ as the set of indices corresponding to the digits of $x$ involved in the assumptions of Corollaries~\ref{stopping condition bloc 1} and \ref{stopping cdt bloc >1} if we want to apply them in the area of the block $B_i$ in $r$. More precisely, if $\ell_i=1$, then 
    \[\textrm{Adm}(i):=[n_i-2, n_i+3].\]
    And, if $\ell_i\ge 2$ then
    \[\textrm{Adm}(i):=[n_i-2, n_i+1] \sqcup [n_i-2+2\ell_i, n_i+1+2\ell_i].\]
    Observe that $\lvert \textrm{Adm}(i)\rvert\le 8$.
    \begin{figure}[H]
        \centering
        \includegraphics[scale=0.35]{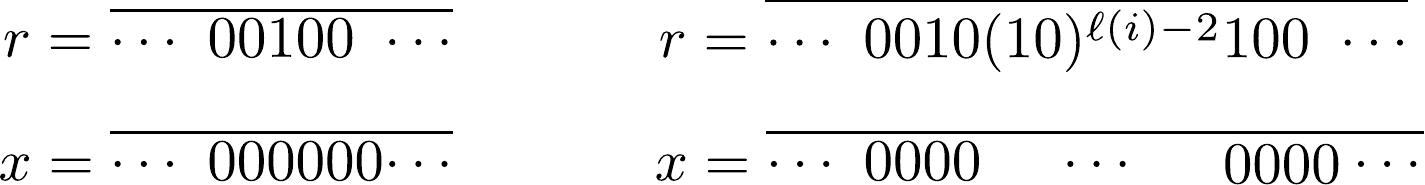}
        \caption{Indices of Adm$(i)$ and conditions put on $x$.}
        \label{Adm(i,l(i))}
    \end{figure}
    Then, we define for $i=1,\cdots,\rho(r)$
    \[C^{(i)}:=\{x\in\XX~: ~\forall j\in \textrm{Adm}(i), ~ x_j=0\}.\]
    We also define the events
    \[ C_1:=\bigcup_{\substack{i=p+2 \\ \textrm{even} }}^{p+\frac{k}{3}}C^{(i)} \hspace{5mm} \textrm{ and } \hspace{5mm} C_2:=\bigcup_{\substack{i=p+\frac{2k}{3} \\ \textrm{even} }}^{p+k-2}C^{(i)}\]
    and, finally 
    \[C:=C_1\cap C_2.\]
    \begin{figure}[H]
        \centering
        \includegraphics[scale=0.35]{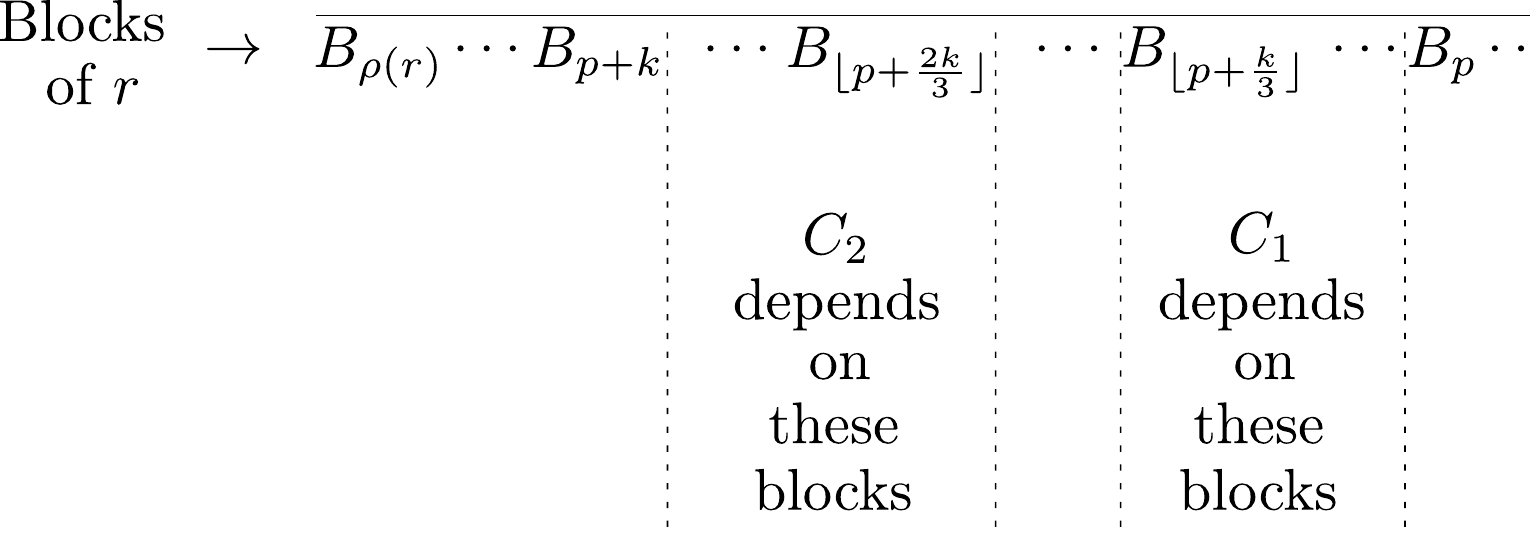}
        \caption{Location of the events $C_1$ and $C_2$ (dots represent blocks or possible zeros).}
        \label{visual}
    \end{figure}
    Thus, $x$ belongs to $C$ if and only if the expansion of $x$ satisfies the hypotheses of Corollary~\ref{stopping condition bloc 1} or \ref{stopping cdt bloc >1} for, at least, two blocks placed on the leftmost third of the window (for at least one block) and the rightmost third of the window (for at least one block).
    
    We prove now that $C$ is an event which has a high probability to happen. We claim that
    \begin{align}\label{majoration C block}
        \PP(C)\ge 1 -2(1-\frac{1}{\varphi^8})^{\frac{k}{6}}.
    \end{align}
    Indeed, denoting $\overline{C_1}$ (resp. $\overline{C_2}$) the complement of $C_1$ (resp. $C_2$) in $\XX$, we have        \[\PP(C)=1-\PP(\overline{C_1})-\PP(\overline{C_2})+\PP(\overline{C_1}\cap\overline{C_2})\ge\PP(C_1)+\PP(C_2)-1.\]
    Then, we obtain from Lemma~\ref{borne proba Zeck}
    \begin{align*}
        \PP(C_1)&=1-\PP\left(\bigcap_{\substack{i=p+2 \\ \textrm{even} }}^{p+\frac{k}{3}} \overline{C^{(i)}}\right) \\
                &=1-\prod_{\substack{i=p+2 \\ \textrm{even} }}^{p+\frac{k}{3}} \PP\left(\overline{C^{(i)}}~\bigg\lvert~\bigcap_{\substack{j=p+2 \\ \textrm{even} }}^{i-2}\overline{C^{(j)}}\right)\\
                & \ge 1-(1-\frac{1}{\varphi^8})^{\frac{k}{6}}
    \end{align*}
    because the product contains $\lfloor\frac{p+\frac{k}{3}-p+2-2}{2}\rfloor=\lfloor\frac{k}{6}\rfloor$ terms. We show that $\PP(C_2$) satisfies the same inequality. We thus obtain \eqref{majoration C block}. Observe that, if $k\ge 194$, $\PP(C)>0$.
    
    Now, we want to control the term $\lvert\PP_C(A\cap B)-\PP_C(A)\PP_C(B)\rvert$ that appears in \eqref{etape 1 mixing block}. We want to use Proposition~\ref{Melange_coord}, and for that we have to clarify which digits of $x$ the events $A$ and $B$ depends on.
    
    But, conditionally to $C$ and thanks to Corollaries~\ref{stopping condition bloc 1} and \ref{stopping cdt bloc >1}, we have that the actions of the $p$ first blocks will only modify digits of indices $\le N_1+2$ where $N_1$ is the index of digit of the leftmost $1$ of the block $B_{\lfloor p+\frac{k}{3}\rfloor}$ of $r$. Thus, there exists $A'\in\sigma(x_i~:~i\le N_1+2)$ such that $A\cap C=A'\cap C$. Likewise, the actions of the $\rho(r)-k-p+1$ last blocks will only modify digits of indices $\ge N_2$ where  $N_2$ the index of the rightmost $0$ of the block $B_{\lfloor p+\frac{2k}{3}\rfloor}$. Thus, there also exists $B'\in\sigma(x_i~:~i\ge N_2)$ such that  $B'\cap C=B\cap C$. 
    \begin{figure}[H]
        \centering
        \includegraphics[scale=0.35]{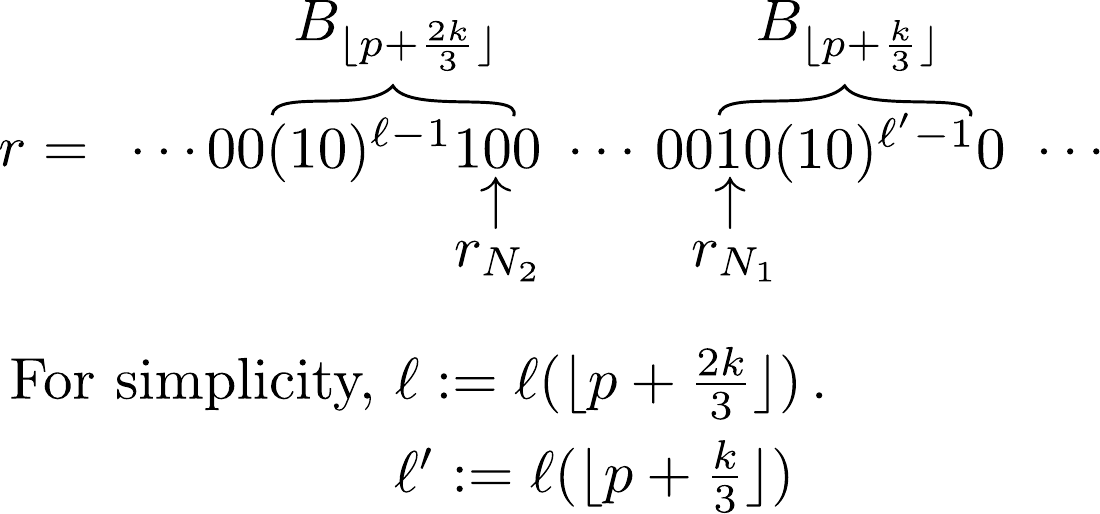}
        \caption{Location of indices $N_1$ and $N_2$ in the expansion of $r$.}
        \label{visual2}
    \end{figure}
    It follows 
     \begin{align}\label{etape 2 mixing block}
        \vert\PP_C(A\cap B)-\PP_C(A)\PP_C(B)\rvert&=\vert\PP_C(A'\cap B')-\PP_C(A')\PP_C(B')\rvert 
    \end{align}
    Then, we get 
    \begin{align*}
        \vert\PP_C(A'\cap B')-\PP_C(A')\PP_C(B')\rvert &\le \vert\PP_C(A'\cap B')-\PP(A'\cap B')\rvert \\
                &\hspace{1cm} +\vert\PP(A'\cap B')-\PP(A')\PP(B')\rvert \\
                &\hspace{1cm} + \vert\PP(A')\PP(B')-\PP_C(A')\PP_C(B')\rvert.
    \end{align*}
    The term $\vert\PP(A'\cap B')-\PP(A')\PP(B')\rvert$ can be controlled using the $\phi$-mixing coefficient on the law of coordinates. Indeed, observe that $N_2-N_1$ is at least about order $k$. Then, using \eqref{ineg generale conditionnement Zeck} and Propositions~\ref{Lien coeff phi et alpha Zeck} and \ref{Melange_coord}, we obtain
    \begin{align}\label{etape 3 mixing block}
        \vert\PP_C(A'\cap B')-\PP_C(A')\PP_C(B')\rvert&\le 3\PP(\overline{C})+\frac{1}{2}\phi(k)
    \end{align}
    Finally, combining \eqref{etape 1 mixing block}, \eqref{etape 2 mixing block} and \eqref{etape 3 mixing block}, we obtain \begin{align*}
         \lvert\PP(A\cap B)-\PP(A)\PP(B)\rvert &\le 6\PP(\overline{C})+\frac{1}{2}\phi(k).
    \end{align*}
    Taking the supremum on $A$ and $B$, we conclude the proof.
\end{proof}

\end{document}